\documentclass[11pt]{amsart}
\usepackage[margin=1.5cm]{geometry}
\usepackage{xcolor}
\usepackage{amsmath}
\usepackage{amsthm}
\usepackage{lipsum}
\usepackage{comment}
\usepackage{amsopn}
\usepackage{amssymb}
\usepackage{fullpage}
\usepackage{soul}
\usepackage{enumerate}
\numberwithin{equation}{section}
\usepackage{float}
\usepackage{graphicx}
\usepackage{titlesec}
\usepackage{dsfont}
\usepackage{fancyhdr}
\usepackage{hyperref}
\usepackage{textcomp}
\usepackage{fancyhdr}
\newcommand{\hiddennumberedsubsection}[1]{%
	\refstepcounter{subsection}%
	\subsectionmark{#1}%
	\noindent\textbf{\thesubsection\quad #1}%
}

\newtheorem{theorem}{Theorem}[section]

\newtheorem{corollary}[theorem]{Corollary}
\newtheorem{lemma}[theorem]{Lemma}
\newtheorem{proposition}[theorem]{Proposition}
\theoremstyle{definition}

\newtheorem{remark}[theorem]{Remark}
\newtheorem{example}[theorem]{Example}

\DeclareMathOperator{\C}{\mathbb{C}}
\DeclareMathOperator{\D}{\mathbb{D}}

\titleformat{\section}
{\normalfont\scshape}{\thesection}{1em}{\centering}
\titleformat{\subsection}[runin]
{\bfseries}{\thesubsection}{1em}{}

\begin{document}
	
	\title[Eigenvalues for Infinitesimal Generators]{Eigenvalues for Infinitesimal Generators of Semigroups of Composition Operators}
	
	\author{Maria Kourou$^1$}
	\thanks{$^1$Partially supported by the Alexander von Humboldt Foundation.}  
	\address{Department of Mathematics, Julius-Maximilians University of W\"urzburg, 97074, W\"urzburg, Germany}
	\email{maria.kourou@uni-wuerzburg.de}
	
	\author{Eleftherios K. Theodosiadis$^2$}
	\thanks{$^2$Supported the Hellenic Foundation for Research and Innovation (H.F.R.I.) under the ``2nd Call for H.F.R.I. Research Projects to support Faculty Members \& Researchers'' (Project Number: 4662).}
	\address{Department of Mathematics, Aristotle University of Thessaloniki, 54124, Thessaloniki, Greece}
	\email{eltheodog@math.auth.gr}

	\author{Konstantinos Zarvalis$^3$}
	\thanks{$^3$Partially supported by Junta de Andaluc\'{i}a, grant number QUAL21 005 USE.}
	\address{Department of Mathematics, Aristotle University of Thessaloniki, 54124, Thessaloniki, Greece}
	\email{zarkonath@math.auth.gr}

	\fancyhf{}
	\renewcommand{\headrulewidth}{0.2pt}
	\fancyhead[RO,LE]{\small \thepage}
	\fancyhead[CE]{\footnotesize Eigenvalues for Infinitesimal Generators}
	\fancyhead[CO]{\footnotesize M. Kourou, E. Theodosiadis, and K. Zarvalis} 
	
	\fancyfoot[L,R,C]{}
	
	\date{}

	\begin{abstract}
		We study the eigenvalues for infinitesimal generators of semigroups of composition operators acting on Hardy spaces, Bergman spaces, and the Dirichlet space. Such semigroups are induced by semigroups of holomorphic functions. Depending on the type of the holomorphic semigroup and the Euclidean geometry of its Koenigs domain, we find containment relations as well as sufficient conditions for the characterization of the point spectrum of the induced infinitesimal generator. For the Dirichlet space we study all types of non-elliptic semigroups whereas for the Hardy and Bergman spaces we work on parabolic semigroups extending the work of Betsakos in the hyperbolic case.
	\end{abstract}
	
	\maketitle
	
	\tableofcontents
	
	\section{\quad Introduction}
	
	\hiddennumberedsubsection{Framework}
	
	Let $X$ be a Banach space of analytic functions in the unit disk $\D$ of the complex plane $\C$. Given a self-map $\phi$ of the unit disk, one may define the \textit{composition operator} $T_\phi$ acting on $X$ with $T_\phi(f):=f\circ\phi$, $f\in X$. The mapping $\phi$ is called the \textit{symbol} of the operator. Composition operators have been a focal point of study in operator theory since the 1960s, with a plethora of high-quality research being published. Various properties of composition operators have been examined, such as the range $T_\phi(X)$, the norm of the operator $T_\phi$, or its point spectrum.
	
	One special instance of composition operators results when the symbols belong to a \textit{one parameter semigroup of holomorphic functions of the unit disk} (from now on a \textit{semigroup in} $\D$). Formally, a semigroup in $\D$ is a family $(\phi_t)$ of holomorphic functions $\phi_t:\D\to\D$, $t\ge0$, which satisfy:
	\begin{enumerate}
		\item[(i)] $\phi_0$ is the identity map in $\D$;
		\item[(ii)] $\phi_{t+s}(z)=\phi_t(\phi_s(z))$, for every $t,s\ge0$ and $z\in\D$;
		\item[(iii)] $\phi_t(z) \xrightarrow{t\to0^+}z$ uniformly on compacta in $\D$.
	\end{enumerate}
	
	Throughout the past two decades, one-parameter semigroups of holomorphic functions have been an area of increased research interest; we refer to the monographs \cite{Booksem, shoikhet} for the complete picture of their theory as well as to some recent works \cite{BZ, contreras_gumenyuk, gumenyuk_kourou_roth}. The modern theory of semigroups emerged in 1978 with the seminal paper of Berkson and Porta \cite{berksonporta}. In their work, they studied semigroups in $\D$ as an aftermath of composition operators. Indeed, given a semigroup $(\phi_t)$ in $\D$ and a Banach space $X$ of analytic functions, one may consider the family of composition operators $(T_t)$, where $T_t:=T_{\phi_t}$, $t\ge0$, acting on $X$. Clearly $T_0$ is the identity operator on $X$ and $T_{t+s}\equiv T_t\circ T_s$, for every $t,s\ge0$. Therefore, the family $(T_t)$ forms a \textit{semigroup of composition operators}. We will say that $(\phi_t)$ \textit{induces} $(T_t)$ or that $(T_t)$ is \textit{induced by} $(\phi_t)$. Semigroups of composition operators have been extensively studied in the past three decades. Siskakis initiated their study in Bergman spaces and the Dirichlet space; see \cite{Siskakis-Bergman} and \cite{Siskakis-Dirichlet}, respectively. Semigroups of composition operators have been examined in various other spaces as well, such as $\textup{Q}_p$ spaces, $\textup{BMOA}$, $\textup{VMOA}$, Bloch spaces, mixed norm spaces, and the disk algebra; see e.g. \cite{contreras_2019, avicou_2016, blasco_etal_2008, chalmoukis_daskalogiannis, CDM-Algebra, contreras_gomez, elin_2018, Qian_Wu_Wulan}. For a review on the theory of semigroups of composition operators on spaces of holomorphic functions in $\D$, we refer the interested reader to \cite{siskakis_review} and references thereafter. 
	
	In the present article, we will deal with the case when $X$ is either the \textit{Dirichlet space} $\mathcal{D}$, or a \textit{(weighted) Bergman space} $A_\alpha^p$, $p\ge1$, $\alpha>-1$, or a \textit{Hardy space} $H^p$, $p\ge1$; detailed information on these spaces follows in Subsection \ref{sub:banachspaces}. In the case of the Dirichlet space, every composition operator induced by a univalent symbol is bounded, while in the Bergman and Hardy spaces \textit{every} composition operator is bounded. In addition, by \cite[Theorem 3.4]{berksonporta}, \cite[Theorem 1]{Siskakis-Bergman}, and \cite[Theorem 1]{Siskakis-Dirichlet} we know that any semigroup of composition operators acting on one of the spaces above is \textit{strongly continuous} (also found as a $C_0$\textit{-semigroup} in the literature) which means that 
	$$\lim\limits_{t\to0^+}||T_t(f)-f||_X=0, \quad\textup{for all }f\in X,$$
	where $||\cdot||_X$ is the norm with respect to the Banach space $X=\mathcal{D},A_\alpha^p,H^p$.
	
	A straightforward aspect of strongly continuous semigroups of composition operators is their \textit{infinitesimal generators}. Consider the set 
	\begin{equation*}
		D_X:=\left\{f\in X: \textup{ there exists }g_f\in X \textup{ with } \lim\limits_{t\to0^+}\left|\left|\dfrac{T_t(f)-f}{t}-g_f\right|\right|_X=0\right\},
	\end{equation*}
	which is dense in $X$. Then, the infinitesimal generator $\Gamma_X$ of $(T_t)$ is the unique (and in general unbounded) operator with $\Gamma_X(f)=g_f=\lim_{t\to0^+}[(T_t(f)-f)/t]$, $f\in X$. Infinitesimal generators of strongly continuous semigroups and their properties have been exhaustively studied throughout the past few years; see e.g. \cite{Chalendar, Gallardo_siskakis, Gallardo_Yakubovich, Reich_Shoikhet}. Among the problems frequently explored in the literature, concerning the theory of strongly continuous semigroups of composition operators, has been the determination of the point spectrum of the infinitesimal generator and its properties. The \textit{point spectrum} of $\Gamma_X$ is defined as the set of eigenvalues of $\Gamma_X$, that is the set of all complex numbers $\lambda$ such that $\Gamma_X(f)=\lambda f$ has non-trivial solutions $f\in X$. For the sake of simplicity, we will denote this point spectrum by $\sigma_X$.
	
	The main objective of the present work is to inspect the set $\sigma_X$ depending on certain properties of the initial holomorphic semigroup $(\phi_t)$ inducing $(T_t)$. Continuous holomorphic semigroups in $\D$ have a strong dynamical property: unless $\phi_{t_0}$ is a conformal automorphism of $\D$ with a fixed point in $\D$, for some $t_0>0$, there exists a unique point $\tau \in \overline{\D}$ such that $\phi_t(z)\xrightarrow{t\to+\infty}\tau$, for all $z\in\D$. This point $\tau$ is called the \textit{Denjoy--Wolff point} of $(\phi_t)$. If $\tau \in \D$, then the semigroup $(\phi_t)$ is called \textit{elliptic}, whereas if it lies on the unit circle we say that the semigroup is \textit{non-elliptic}. In the case where $\tau \in \partial \D$ and the angular derivative $\phi_1^{\prime}(\tau)\in(0,1)$, the semigroup $(\phi_t)$ is called \textit{hyperbolic}, while if $\phi_1^{\prime}(\tau)=1$, $(\phi_t)$ is called \textit{parabolic}. We will only work with non-elliptic semigroups because whenever the Denjoy--Wolff point lies on $\partial\D$, the dynamic behavior of the semigroup presents more intricacies and renders its investigation more interesting from a geometric perspective. For each non-elliptic semigroup $(\phi_t)$ there exists a unique (up to translation), univalent mapping $h:\D\to\C$ such that
	\begin{equation}\label{eq:koenigs}
		h(\phi_t(z))=h(z)+t, \quad \textup{for all }t\ge0\textup{ and all }z\in\D.
	\end{equation}
	This function is called the \textit{Koenigs function} of $(\phi_t)$ and the simply connected domain $\Omega:=h(\D)$ its \textit{Koenigs domain}. Our goal is to inspect the point spectrum $\sigma_X$ with respect to the type of the semigroup $(\phi_t)$ and the intrinsic geometric properties of $\Omega$. For this reason, we will use a function that fully describes this geometry, the \textit{defining function} of $(\phi_t)$. Given a non-elliptic semigroup $(\phi_t)$ in $\D$, its defining function is the unique upper semi-continuous function $\psi:I\to [-\infty,+\infty)$, where $I\subseteq\mathbb{R}$ is an open interval, such that the Koenigs domain $\Omega$ can be parametrized as 
	\begin{equation}\label{eq:defining function}
		\Omega=\{x+iy:y\in I, \; x>\psi(y)\}.
	\end{equation}
	This function was introduced in \cite[Definition 2.3]{BGGY}. In particular, excluding the case $I=\mathbb{R}$ and $\psi\equiv -\infty$ which results in the whole complex plane, the authors prove in \cite[Proposition 2.2]{BGGY} that there exists a one-to-one correspondence between non-elliptic Koenigs domains $\Omega$ and upper semi-continuous functions $\psi$ defined on open intervals. Therefore, the defining function conceals the complete nature of $\Omega$. On that account, henceforth we will denote the defining function of $(\phi_t)$ by $\psi_\Omega$ to emphasize its inextricable connection to the Koenigs domain.

	\hiddennumberedsubsection{Statement of Results}
	
	We start with hyperbolic semigroups and the Dirichlet space aiming to find results akin to those in \cite[Theorem 3]{Betsakos_Bergman} and \cite[Theorem 1]{betsakos_eig} for the Bergman and Hardy spaces, respectively. In his works, Betsakos obtained a complete characterization of the point spectrum only taking under consideration the width of the smallest horizontal strip containing $\Omega$ and the width of the largest horizontal strip contained inside $\Omega$. Hence, a multitude of hyperbolic semigroups have exactly the same point spectrum in Bergman and Hardy spaces, without taking into account the exact geometry of $\Omega$. For the Dirichlet space, the situation is much different. We will see that the widths of the aforementioned strips do not factor in the point spectrum. On the contrary, the total geometry of the Koenigs domain counts, presenting a firmer connection with the corresponding point spectrum. Our first main result is the following:
	\begin{theorem}\label{thm:intro defining function}
		Let $(\phi_t)$ be a hyperbolic semigroup in $\D$ which induces the semigroup of composition operators $(T_t)$. Suppose that $\Omega$ is the Koenigs domain of $(\phi_t)$ and $\psi_\Omega:(a,b)\to[-\infty,+\infty)$ its defining function, where $-\infty<a<b<+\infty$. If $\sigma_\mathcal{D}$ is the point spectrum of the infinitesimal generator of $(T_t)$ when acting on the Dirichlet space, then:
		\begin{enumerate}
			\item[\textup{(a)}] It holds that $\lambda\in\sigma_\mathcal{D}\setminus\{0\}$ if and only if $\int_{a}^{b}e^{2\mathrm{Re}\lambda\psi_\Omega(y)}dy<+\infty$.
		\end{enumerate}
		Set $c_\Omega=\inf\{c\in(-\infty,0]:\int_{a}^{b}e^{2c\psi_\Omega(y)}dy<+\infty\}$. Then:
		\begin{enumerate}
			\item[\textup{(b)}]  If $c_\Omega=0$, then $\sigma_\mathcal{D}=\{0\}$.
			\item[\textup{(c)}] If $c_\Omega\in(-\infty,0)$ and $\int_{a}^{b}e^{2c_\Omega\psi_\Omega(y)}dy=+\infty$, then $\sigma_\mathcal{D}=\{\lambda\in\C:c_\Omega<\mathrm{Re}\lambda<0\}\cup\{0\}$.
			\item[\textup{(d)}] If $c_\Omega\in(-\infty,0)$ and $\int_{a}^{b}e^{2c_\Omega\psi_\Omega(y)}dy<+\infty$, then $\sigma_\mathcal{D}=\{\lambda\in\C:c_\Omega\le\mathrm{Re}\lambda<0\}\cup\{0\}$.
			\item[\textup{(e)}] If $c_\Omega=-\infty$, then $\sigma_\mathcal{D}=\{\lambda\in\C:-\infty<\mathrm{Re}\lambda<0\}\cup\{0\}$.
		\end{enumerate}
	\end{theorem}
	
	Through Theorem \ref{thm:intro defining function}, we obtain a complete characterization of the point spectrum. In fact, the point spectrum excluding the eigenvalue $0$ is either the left half-plane, or a vertical strip, or empty. However, in certain cases, the defining function can be extremely complicated rendering the actual calculation of the integrals above difficult. So, using a different parametrization of $\Omega$, we will also prove a more handy, but this time partial result, which enables us to easily calculate the point spectrum in the Dirichlet space for a wide class of hyperbolic semigroups, making its relation with the geometry of $\Omega$ more transparent. Let $(\phi_t)$ be a hyperbolic semigroup with Koenigs domain $\Omega$. For $x\in\mathbb{R}$ set $\Omega_x:=\{y\in\mathbb{R}:x+iy\in\Omega\}$. Denote by $\ell_\Omega(x)$ the Euclidean length of $\Omega_x$. Due to \cite[Theorem 9.3.5]{Booksem}, we know that $\Omega$ is contained in a minimum horizontal strip $\Sigma$, and thus $\ell_\Omega(x)$ is uniformly bounded from above by the width of $\Sigma$. For the statement of our result we will need some notation. Set 
	\begin{equation}\label{eq:intro hyperbolic dirichlet}
		W(\Omega):=\int\limits_{-\infty}^{0}\ell_\Omega(x)dx=\textup{Area}(\Omega\cap\{w\in\C:\mathrm{Re}w<0\}) \quad \textup{and} \quad\delta_\Omega:=\lim\limits_{x\to-\infty}\dfrac{\log\ell_\Omega(x)}{2x},
	\end{equation}
	assuming that the latter limit exists. The two quantities defined in \eqref{eq:intro hyperbolic dirichlet} are a measure of the area of $\Omega$ and the speed at which the boundary $\partial\Omega$ diverges from $\partial \Sigma$, as $x\to-\infty$. As a matter of fact, there are examples where the limit $\delta_\Omega$ does not exist; see Example \ref{ex:limit does not exist}. In such cases, the boundary $\partial\Omega$ moves away from $\partial\Sigma$ in a more pathological way. Our second result for the point spectrum $\sigma_\mathcal{D}$ in the Dirichlet space, induced by a hyperbolic semigroup, is the following:
	
	\begin{theorem}\label{thm:intro hyperbolic dirichlet}
		Let $(\phi_t)$ be a hyperbolic semigroup in $\D$ which induces the semigroup of composition operators $(T_t)$. Suppose that $\Omega$ is the Koenigs domain of $(\phi_t)$ and $\sigma_\mathcal{D}$ the point spectrum of the infinitesimal generator of $(T_t)$. Then:
		\begin{enumerate}
			\item[\textup{(a)}] If $W(\Omega)=+\infty$, then $\sigma_{\mathcal{D}}=\{0\}$.
			\item[\textup{(b)}] If $W(\Omega)<+\infty$ and $\delta_\Omega=+\infty$, then $\sigma_{\mathcal{D}}=\{\lambda\in\C:\mathrm{Re}\lambda<0\}\cup\{0\}$.
			\item[\textup{(c)}] If $W(\Omega)<+\infty$, $\delta_\Omega\in(0,+\infty)$, and $\int_{-\infty}^{0}e^{-2\delta_\Omega x}\ell_\Omega(x)dx<+\infty$, then
			$$\sigma_{\mathcal{D}}=\{\lambda\in\C:-\delta_\Omega\le\mathrm{Re}\lambda<0\}\cup\{0\}.$$
			\item[\textup{(d)}] If $W(\Omega)<+\infty$, $\delta_\Omega\in(0,+\infty)$, and $\int_{-\infty}^{0}e^{-2\delta_\Omega x}\ell_\Omega(x)dx=+\infty$, then
			$$\sigma_{\mathcal{D}}=\{\lambda\in\C:-\delta_\Omega<\mathrm{Re}\lambda<0\}\cup\{0\}.$$
			\item [\textup{(e)}] If $W(\Omega)<+\infty$ and $\delta_\Omega=0$, then $\sigma_{\mathcal{D}}=\{0\}$.
		\end{enumerate}
	\end{theorem}
	If the limit $\delta_\Omega$ does not exist (something that can happen as seen in an explicit example later on), then taking the limit inferior and the limit superior, similar techniques provide useful inclusions for the point spectrum. As a result, in general, Theorem \ref{thm:intro hyperbolic dirichlet} provides a ``blueprint'' through which one knows ``where to look'' in order to explicitly compute the point spectrum.
	
	We move on to parabolic semigroups and begin with a distinction among the class of parabolic semigroups. Let $(\phi_t)$ be parabolic with Koenigs domain $\Omega$. If $\Omega$ is contained in a horizontal half-plane, we say that $(\phi_t)$ is \textit{parabolic of positive hyperbolic step}. If not, then $(\phi_t)$ is called \textit{parabolic of zero hyperbolic step}. Contrary to the hyperbolic case, we will now examine the point spectrum under the scope of the angular sectors containing $\Omega$ or contained inside $\Omega$. Given $a,b\in[-\pi,\pi]$ with $a<b$, we set $S(a,b)=\{w\in\C:a<\arg w<b\}$, $S[a,b)=\{w\in\C: a\le \arg w<b\}$, and $S(a,b]=\{w\in\C: a<\arg w\le b\}$. Suppose that $(\phi_t)$ is parabolic of positive hyperbolic step. Then, through a translation, we can always assume that $\Omega$ is contained in the usual upper half-plane $S(0,\pi)$. In this case, we introduce two intrinsic geometric quantities that describe the shape of $\Omega$. First, the \textit{inner argument} of $(\phi_t)$ is defined as
	\begin{equation}\label{eq:inner argument}
		\theta_\Omega:=\sup\left\{\{0\}\cup\{\theta\in(0,\pi]: \textup{there exists }q_\theta\in\C \textup{ such that }q_\theta+S(0,\theta)\subseteq\Omega\}\right\}.
	\end{equation}
	In a similar vein, the \textit{outer argument} of $(\phi_t)$ is
	\begin{equation}\label{eq:outer argument}
		\Theta_\Omega:=\inf\{\theta\in(0,\pi]: \textup{there exists }q_\theta\in\C \textup{ such that }\Omega\subseteq q_\theta+S(0,\theta)\}.
	\end{equation}
	Clearly $\theta_\Omega\le\Theta_\Omega$. For more information in these quantities, we refer to \cite[Section 3]{our_finiteshift}. Our main result is the following:
	
	\begin{theorem}\label{thm:intro phs dirichlet}
		Let $(\phi_t)$ be a parabolic semigroup of positive hyperbolic step in $\D$ which induces the semigroup of composition operators $(T_t)$. Suppose that $\Omega$ is the Koenigs domain of $(\phi_t)$ and that $\sigma_\mathcal{D}$ is the point spectrum of the infinitesimal generator of $(T_t)$. Then:
		\begin{enumerate}
			\item[\textup{(a)}] If $\theta_\Omega=\Theta_\Omega=0$, then $\sigma_\mathcal{D}=iS(0,\pi)\cup\{0\}$.
			\item[\textup{(b)}] If $\theta_\Omega=\Theta_\Omega\in(0,\pi)$, then $iS(0,\pi-\theta_\Omega)\cup\{0\}\subseteq\sigma_\mathcal{D}\subseteq iS(0,\pi-\theta_\Omega]\cup\{0\}$.
			\item[\textup{(c)}] If $\theta_\Omega=\Theta_\Omega=\pi$, then $\sigma_\mathcal{D}=\{0\}$.
			\item[\textup{(d)}] If $\theta_\Omega<\Theta_\Omega$, then $iS(0,\pi-\Theta_\Omega)\cup\{0\}\subseteq\sigma_\mathcal{D}\subseteq iS(0,\pi-\theta_\Omega]\cup\{0\}$. In case $\Theta_\Omega=\pi$, the set on the left becomes $\{0\}$, while if $\theta_\Omega=0$, the set on the right becomes $iS(0,\pi)\cup\{0\}$.
		\end{enumerate}
	\end{theorem}
	Of course, if $\Omega$ was contained in any other upper or lower half-plane, our result still holds, albeit with minor obvious modifications.
	
	Finally, for parabolic semigroups of zero hyperbolic step, the Koenigs domain is no longer contained in a horizontal half-plane and we are in need of a generalization of the inner and outer arguments. Once again, through a translation we may always assume that $\Omega\subseteq\C\setminus\{-t:t\ge0\}$. This is because the Koenigs domain satisfies $\Omega+t\subseteq\Omega$, for all $t\ge0$, and must also be a proper subdomain of the whole complex plane, due to the definition of $h$. Clearly the sets $\Omega^-:=\Omega\cap S(-\pi,0)$ and $\Omega^+:=\Omega\cap S(0,\pi)$ are connected and non-empty. Then, we may define the \textit{lower inner argument} of $(\phi_t)$ as
	\begin{equation}\label{eq:lower inner argument}
		\theta_\Omega^-:=\sup\left\{\{0\}\cup\{\theta\in(0,\pi]:\textup{there exists }q_\theta\in\C \textup{ such that }q_\theta+S(-\theta,0)\subseteq\Omega^-\}\right\},
	\end{equation}
	and the \textit{lower outer argument} as
	\begin{equation}\label{eq:lower outer argument}
		\Theta_\Omega^-:=\inf\{\theta\in(0,\pi]:\textup{there exists }q_\theta\in\C \textup{ such that }\Omega^-\subseteq q_\theta+S(-\theta,0)\}.
	\end{equation}

	In similar fashion, using the set $\Omega^+$, we can define the \textit{upper inner argument} $\theta_\Omega^+$ and the \textit{upper outer argument} $\Theta_\Omega^+$; see Figure \ref{fig:Everything/arguments.png}. Through these four numbers, we may generalize the notions of inner and outer arguments. Indeed, for parabolic semigroups of zero hyperbolic step, we may set $\theta_\Omega:=\theta_\Omega^-+\theta_\Omega^+$ and $\Theta_\Omega:=\Theta_\Omega^-+\Theta_\Omega^+$. We will prove the following:
	\begin{figure}[ht]
		\centering
		\includegraphics[width=0.8\linewidth]{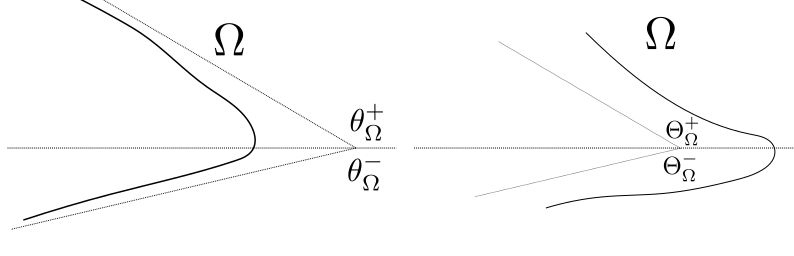}
		\caption{The inner/outer, lower and upper arguments, respectively.}
		\label{fig:Everything/arguments.png}
	\end{figure}
	\begin{theorem}\label{thm:intro 0hs dirichlet}
		Let $(\phi_t)$ be a parabolic semigroup of zero hyperbolic step in $\D$ which induces the semigroup of composition operators $(T_t)$. Suppose that $\Omega$ is the Koenigs domain of $(\phi_t)$ and that $\sigma_\mathcal{D}$ is the point spectrum of the infinitesimal generator of $(T_t)$. Then:
		\begin{enumerate}
			\item[\textup{(a)}] If $\theta_\Omega=\Theta_\Omega=0$, then $\sigma_\mathcal{D}=iS(0,\pi)\cup\{0\}$.
			\item[\textup{(b)}] If $\theta_\Omega=\Theta_\Omega\in(0,\pi)$ (and hence $\theta_\Omega^-=\Theta_\Omega^-$, $\theta_\Omega^+=\Theta_\Omega^+$), then
			$$iS(\theta_\Omega^-,\pi-\theta_\Omega^+)\cup\{0\}\subseteq\sigma_\mathcal{D}\subseteq i\overline{S(\theta_\Omega^-,\pi-\theta_\Omega^+)}.
			$$
			\item[\textup{(c)}] If $\theta_\Omega=\pi$, then $\sigma_\mathcal{D}\subseteq\{ite^{i\theta_\Omega^-}:t\ge0\}$.
			\item[\textup{(d)}] If $\theta_\Omega\in(\pi,2\pi]$, then $\sigma_\mathcal{D}=\{0\}$.
			\item[\textup{(e)}] If $\theta_\Omega<\min\{\pi,\Theta_\Omega\}$, then $iS(\Theta_\Omega^-,\pi-\Theta_\Omega^+)\cup\{0\}\subseteq\sigma_\mathcal{D}\subseteq i\overline{S(\theta_\Omega^-,\pi-\theta_\Omega^+)}$. In case $\Theta_\Omega\ge\pi$, the set on the left becomes $\{0\}$, while if $\theta_\Omega=0$, the set on the right becomes $iS(0,\pi)\cup\{0\}$.
		\end{enumerate}
	\end{theorem}
	
	Although, our results on parabolic semigroups in the setting of the Dirichlet space do not yield a precise characterization of the point spectrum, they do pose as a helpful guide in order to locate it. Note that Theorem \ref{thm:intro phs dirichlet} provides a little more insight than Theorem \ref{thm:intro 0hs dirichlet}, since the inclusions on the right are stricter. This will be the outcome of a ``convexity'' result we will prove about the point spectrum. Besides, due to the \textit{growth bounds} of semigroups of composition operators induced by parabolic semigroups, it is known that the point spectrum lies in the closed left half-plane and hence, our results provide an even stricter inclusion. 
	
	After dealing with the Dirichlet space, we prove the counterparts of Theorem \ref{thm:intro phs dirichlet} and \ref{thm:intro 0hs dirichlet} for the Bergman spaces and the Hardy spaces. Recall that the counterpart of Theorem \ref{thm:intro hyperbolic dirichlet} for these spaces has already been dealt with by Betsakos. As mentioned before, the growth bound of semigroups of composition operators in these spaces yields that the point spectrum lies in the closed left half-plane. In this case, the inclusion is sharp. For the sake of brevity, we write $\sigma_{p,\alpha}:=\sigma_{A_\alpha^p}(\Gamma_{A_\alpha^p})$ and $\sigma_p:=\sigma_{H^p}(\Gamma_{H^p})$. Through the \textit{Littlewood-Paley Formula} and the \textit{Hardy-Stein Formula}, our techniques involve potential theory and allow us to treat the Hardy space $H^p$, $p\ge1$ as an ``extremal'' Bergman space $A_{-1}^p$. As a matter of fact, we prove the following two results:
	
	\begin{theorem}\label{thm:intro bergman phs}
		Let $(\phi_t)$ be a parabolic semigroup of positive hyperbolic step in $\D$ which induces the semigroup of composition operators $(T_t)$. Suppose that $\Omega$ is the Koenigs domain of $(\phi_t)$ and that $\sigma_{p,\alpha}$ is the point spectrum of the infinitesimal generator of $(T_t)$ when acting on the Bergman space $A_\alpha^p$, $p\ge1$, $\alpha>-1$. Then:
		\begin{enumerate}
			\item[\textup{(a)}] If $\theta_\Omega=\Theta_\Omega=0$, then $iS[0,\pi)\cup\{0\}\subseteq\sigma_{p,\alpha}\subseteq i\overline{S(0,\pi)}$.
			\item[\textup{(b)}] If $\theta_\Omega=\Theta_\Omega\in(0,\pi)$, then $iS[0,\pi-\theta_\Omega)\cup\{0\}\subseteq \sigma_{p,\alpha} \subseteq i\overline{S(0,\pi-\theta_\Omega)}$.
			\item[\textup{(c)}] If $\theta_\Omega=\Theta_\Omega=\pi$, then $\sigma_{p,\alpha}=\{it:t\ge0\}$.
			\item[\textup{(d)}] If $\theta_\Omega<\Theta_\Omega$, then $iS[0,\pi-\Theta_\Omega)\cup\{0\}\subseteq\sigma_{p,\alpha}\subseteq i\overline{S(0,\pi-\theta_\Omega)}$. In case $\Theta_\Omega=\pi$, the set on the left becomes the vertical half-line $\{it:t\ge0\}$, while if $\theta_\Omega=0$, the set on the right becomes $i\overline{S(0,\pi)}$.
		\end{enumerate}
		The same results hold for the point spectrum $\sigma_p$ when $(T_t)$ acts on the Hardy space $H^p$, $p\ge1$.
	\end{theorem}
	
	\begin{theorem}\label{thm:intro bergman 0hs}
		Let $(\phi_t)$ be a parabolic semigroup of zero hyperbolic step in $\D$ which induces the semigroup of composition operators $(T_t)$. Suppose that $\Omega$ is the Koenigs domain of $(\phi_t)$ and that $\sigma_{p,\alpha}$ is the point spectrum of the infinitesimal generator of $(T_t)$ when acting on the Bergman space $A_\alpha^p$, $p\ge1$, $\alpha>-1$. Then:
		\begin{enumerate}
			\item[\textup{(a)}] If $\theta_\Omega=\Theta_\Omega=0$, then $iS(0,\pi)\cup\{0\}\subseteq\sigma_{p,\alpha}\subseteq i\overline{S(0,\pi)}$.
			\item[\textup{(b)}] If $\theta_\Omega=\Theta_\Omega\in(0,\pi)$ (and hence $\theta_\Omega^-=\Theta_\Omega^-, \theta_\Omega^+=\Theta_\Omega^+)$, then
			$$iS(\theta_\Omega^-,\pi-\theta_\Omega^+)\cup\{0\}\subseteq \sigma_{p,\alpha} \subseteq i\overline{S(\theta_\Omega^-,\pi-\theta_\Omega^+)}.$$
			\item[\textup{(c)}] If $\theta_\Omega=\pi$, then $\sigma_{p,\alpha}\subseteq \{ie^{i\theta_\Omega^-}t:t\ge0\}$.
			\item[\textup{(d)}] If $\theta_\Omega\in(\pi,2\pi]$, then $\sigma_{p,\alpha}=\{0\}$.
			\item[\textup{(e)}] If $\theta_\Omega<\min\{\pi,\Theta_\Omega\}$, then $iS(\Theta_\Omega^-,\pi-\Theta_\Omega^+)\cup\{0\}\subseteq\sigma_{p,\alpha}\subseteq i\overline{S(\theta_\Omega^-,\pi-\theta_\Omega^+)}$. In case $\Theta_\Omega\ge\pi$, the set on the left becomes $\{0\}$, while if $\theta_\Omega=0$, the set on the right becomes $i\overline{S(0,\pi)}$.
		\end{enumerate} 
		The same results hold for the point spectrum $\sigma_p$ when $(T_t)$ acts on the Hardy space $H^p$, $p\ge1$.
	\end{theorem}
	
	In all of the aforementioned spaces, we first explicitly compute the point spectrum whenever the Koenigs domain is exactly an angular sector. Using this as a stepping stone, we then prove our results using monotonicity properties of the point spectrum with regard to the Koenigs domain. Moreover, for the Bergman and Hardy spaces, we use careful estimates of the Green's function on $\Omega$.

	The article is structured as follows: Section \ref{sec:prereq} constitutes a collection of all useful tools that are required for the proofs. Next, in Section \ref{sec:dirichlet} we examine the semigroups of composition operators in the Dirichlet space and prove Theorems \ref{thm:intro defining function}, \ref{thm:intro hyperbolic dirichlet}, \ref{thm:intro phs dirichlet} and \ref{thm:intro 0hs dirichlet}. Then, we take a brief detour and study the growth bounds and norms of the operators in Section \ref{sec:growth bounds}, using hyperbolic geometry. Finally, Section \ref{sec:bergman and hardy} is the part of the manuscript that contains the proofs of Theorems \ref{thm:intro bergman phs} and \ref{thm:intro bergman 0hs}, along with other results describing the nature and form of the point spectrum in Bergman and Hardy spaces. During the course of the article, we prove certain corollaries about the point spectra in conjunction with the dynamical behavior of $(\phi_t)$. For instance, in all the spaces, we characterize completely the point spectrum when the inducing semigroup is of \textit{finite shift}. To better demonstrate the utility of our results, we offer various non-trivial examples where the study of the point spectrum may be facilitated through our theorems and computed through the Euclidean geometry of the Koenigs domain.

	\section{\quad Prerequisites for the Proofs}\label{sec:prereq}
	\hiddennumberedsubsection{Holomorphic Semigroup Theory}\label{sub:semigroups}
	We start with some background on continuous semigroups of holomorphic functions. For this subsection, we almost exclusively follow \cite{Booksem}. Suppose that $(\phi_t)$ is a non-elliptic semigroup in $\D$ with Denjoy--Wolff point $\tau\in\partial\D$, Koenigs function $h$ and Koenigs domain $\Omega$. Firstly, given $z\in\D$, the set $\{\phi_t(z):t\ge0\}$ is known as the \textit{orbit} of $z$ and $z$ is its \textit{starting point}. Furthermore, there exists a unique $\mu\ge0$ such that $ \angle\lim_{z\to\tau}\phi_t'(z)=e^{-\mu t}$, for all $t\ge0$. In both cases, the unique number $\mu$ is called the \textit{spectral value} of $(\phi_t)$. As mentioned in the Introduction, we know that if $\mu>0$, then $(\phi_t)$ is hyperbolic, whereas if $\mu=0$, then it is parabolic.
	
	\textup{From now on, we write $d_{\D}(z,w)$ for the hyperbolic distance in the unit disk of two points $z,w\in\D$. Detailed information on hyperbolic geometry follows in Subsection \ref{sub:conformal invariants}.} A useful result with regard to the spectral value of non-elliptic semigroups is the following:
	\begin{theorem}[{\cite[Theorem 9.1.9]{Booksem}}]\label{thm:total speed}
		Let $(\phi_t)$ be a non-elliptic semigroup in $\D$ with spectral value $\mu\ge0$. Then, for all $z\in\D$ we have
		$$\lim\limits_{t\to+\infty}\dfrac{d_{\D}(z,\phi_t(z))}{t}=\dfrac{\mu}{2}.$$
	\end{theorem}
	Besides, there exists a unique holomorphic function $G:\D\to\C$ such that
	\begin{equation*}
		\dfrac{\partial\phi_t(z)}{\partial t}=G(\phi_t(z)), \quad\textup{for all }t\ge0 \textup{ and all }z\in\D.	
	\end{equation*}
	This function is called the \textit{infinitesimal generator} of $(\phi_t)$ and is the counterpart of $\Gamma_X$ for the induced semigroup $(T_t)$. In particular, there is a useful equality relating the infinitesimal generator $G$ with the Koenigs function $h$. Indeed, whenever $(\phi_t)$ is non-elliptic, we have
	\begin{equation}\label{eq:generator non-elliptic}
		G(z)=\dfrac{1}{h'(z)}, \quad\textup{for all }z\in\D.
	\end{equation}
	
	Concerning the Koenigs domain, the type and spectral value of a non-elliptic semigroup dictate (up to a certain point) the geometry of $\Omega$. As a matter of fact, $\Omega$ is \textit{convex in the positive direction}, i.e. $\Omega+t\subseteq\Omega$, for all $t\ge0$. In particular, if $(\phi_t)$ is a hyperbolic with spectral value $\mu>0$, then there exists a minimal strip of width $\pi/\mu$ containing $\Omega$. If $(\phi_t)$ is parabolic of positive hyperbolic step, then $\Omega$ is contained in some horizontal half-plane, but not inside a horizontal strip. Finally, if $(\phi_t)$ is parabolic of zero hyperbolic step, then $\Omega$ is not contained in any horizontal half-plane. Combining with the fact that $\Omega$ is convex in the positive direction and not the whole complex plane, we get that $\Omega$ is contained in some slit plane. 
	
	The quantity that downright describes the geometry of $\Omega$ is the so-called \textit{defining function} of $(\phi_t)$; see \cite[Definition 2.3]{BGGY}. As we mentioned before, given a non-elliptic semigroup $(\phi_t)$ in $\D$, there exists a unique upper semi-continuous function $\psi_\Omega:I\to[-\infty,+\infty)$, where $I$ is an open subinterval of $\mathbb{R}$ such that the Koenigs domain $\Omega$ may be parametrized as $\Omega=\{x+iy\in\C:y\in I,\; x>\psi_\Omega(y)\}$. Note that we allow the function $\psi_\Omega$ to take the value $-\infty$. Recall that if $I$ is an open subinterval of $\mathbb{R}$, a function $u:I\to[-\infty,+\infty)$ is called \textit{upper semi-continuous} provided that for every $x_0\in I$
	\begin{equation*}
		\limsup_{x\to x_0}u(x)\le u(x_0).
	\end{equation*}
	Other than our first classification of semigroups, we may proceed to an entirely different one. Let $(\phi_t)$ be a non-elliptic semigroup in $\D$ with Denjoy--Wolff point $\tau\in\partial\D$. The new classification relies on the notion of horodisks. A \textit{horodisk} centered at $\tau$ is a disk contained inside $\D$ which is internally tangent to $\partial\D$ at $\tau$. Then, $(\phi_t)$ is said to be of \textit{finite shift} if there exists some $z\in\D$ such that its orbit converges to $\tau$ without ever intersecting some horodisk of $\D$ centered at $\tau$. It is known by \cite[Lemma 17.7.4]{Booksem} that if this behavior is demonstrated for some $z\in\D$, then it is shared by all points in the unit disk. On the other hand, we say that $(\phi_t)$ is of \textit{infinite shift} if every orbit intersects all horodisks centered at $\tau$.
	
	We will only need certain basic facts about semigroups of finite shift. First of all, such semigroups are necessarily parabolic of positive hyperbolic step. In particular, in \cite[Proposition 3.2]{our_finiteshift} it is shown that they always have inner (and by extension outer) argument equal to $\pi$. Finally, the following result about the asymptotic behavior of the hyperbolic distance along an orbit holds:
	\begin{lemma}[{\cite[Corollary 4.1]{our_finiteshift}}]\label{lm:total speed finite shift}
		Let $(\phi_t)$ be a semigroup of finite shift in $\D$. Then, for all $z\in\D$ we have
		$$\lim\limits_{t\to+\infty}\dfrac{d_{\D}(z,\phi_t(z))}{\log t}=1.$$
	\end{lemma}
	Note that Lemma \ref{lm:total speed finite shift} improves Theorem \ref{thm:total speed} for semigroups of finite shift.
	
	\hiddennumberedsubsection{Certain Banach Spaces}\label{sub:banachspaces}
	In this subsection, we will review some fundamental facts about the spaces of analytic functions that we are going to use in the present work. As we said during the Introduction, we will work with the Hardy spaces, the Bergman spaces and the Dirichlet space. 
	
	We start with the classical Hardy spaces $H^p$, $p\ge1$, of the unit disk. For $p\in[1,+\infty)$, the \textit{Hardy space} $H^p$ is the space of all analytic functions $f:\D\to\C$ satisfying
	\begin{equation}\label{eq:hardy definition}
		||f||_{H^p}:=\sup\limits_{r\in[0,1)}\left(\int\limits_{0}^{2\pi}|f(re^{i\theta})|^pd\theta\right)^{1/p}<+\infty.
	\end{equation}
	For $p=+\infty$, the Hardy space $H^\infty$ is the space of all bounded analytic functions in the unit disk. A very useful characterization about the membership of a holomorphic function in a Hardy space is the well-known \textit{Littlewood-Paley Formula}. Note that by $dA$ we mean the normalized Lebesgue area measure.
	\begin{theorem}[{cf. \cite[Theorem 1]{Yamashita}}]\label{thm:hardy hardy-stein}
		Let $f:\D\to\C$ be holomorphic and $p\in[1,+\infty)$. Then $f\in H^p$ if and only if
		\begin{equation}\label{eq:hardy hardy-stein}
			\int\limits_{\D}|f(z)|^{p-2}|f'(z)|^2\log\frac{1}{|z|}dA(z)<+\infty.
		\end{equation}
	\end{theorem}
	Note that in \cite[Theorem 1]{Yamashita} the factor $\log(1/|z|)$ in the integral is not present and in its place there is the factor $1-|z|$. It is easily verified that the two results are equivalent.
	
	Moving on, let $p\in[1,+\infty)$ and $\alpha\in(-1,+\infty)$. Then, the \textit{(weighted) Bergman space} $A_\alpha^p$ consists of all functions $f$ holomorphic in the unit disk that satisfy
	\begin{equation}\label{eq:bergman norm}
		||f||_{A_\alpha^p}:=\left(\int\limits_{\D}|f(z)|^p(1-|z|^2)^\alpha dA(z)\right)^{1/p}<+\infty.
	\end{equation}
	Similarly to the Hardy spaces, there exists a handy integrability condition about the membership of a holomorphic function in a Bergman space. This is known as the \textit{Hardy-Stein Formula}.
	\begin{theorem}[{\cite[Lemma 2.3]{Smith}}]\label{thm:bergman hardy-stein}
		Let $f:\D\to\C$ be holomorphic, $p\in[1,+\infty)$ and $\alpha\in(-1,+\infty)$. Then $f\in A_\alpha^p$ if and only if
		\begin{equation}\label{eq:bergman hardy-stein}
			\int\limits_{\D}|f(z)|^{p-2}|f'(z)|^2\left(\log\frac{1}{|z|}\right)^{\alpha+2}dA(z)<+\infty.
		\end{equation}
	\end{theorem}
	
	Taking into account both Theorems \ref{thm:hardy hardy-stein} and \ref{thm:bergman hardy-stein}, we could say that each Hardy space $H^p$, $p\in[1,+\infty)$, can be thought of as the space $A_{-1}^p$. This observation will play a key role into accelerating the proofs when working in the setting of the Hardy spaces. Of course, the Hardy and Bergman spaces may be also defined in the same way for $p\in(0,1)$ and Theorems \ref{thm:hardy hardy-stein} and \ref{thm:bergman hardy-stein} still hold. However, for this choice of $p$, they cease to be Banach spaces.
	
	Finally, the \textit{Dirichlet space} $\mathcal{D}$ is the space of all holomorphic functions $f$ in the unit disk such that
	\begin{equation}
		\mathcal{D}(f):=\int\limits_{\D}|f'(z)|^2dA(z)<+\infty.
	\end{equation}
	Note that $\mathcal{D}(f)$ is a semi-norm. Through the change of variables formula, we understand that a holomorphic $f:\D\to\C$ belongs to the Dirichlet space if and only if the image $f(\D)$ has bounded area counting multiplicity.

	\hiddennumberedsubsection{Operator Theory}\label{sub:operator theory}
	We start with certain basic notions with regard to semigroups of composition operators. For an in-depth presentation of their rich theory, we refer to the survey \cite{siskakis_review} and the books \cite{Cowen-Maccluer, zhu}. Let $X$ be a Banach space of analytic functions in the unit disk such that $f\circ \phi\in X$, for any choice of $f\in X$ and injective self-map $\phi$ of $\D$. Suppose that $(\phi_t)$ is a semigroup of holomorphic functions in $\D$. For each $t\ge0$, consider the composition operator $T_t$ acting on $X$, where $T_t(f)=f\circ\phi_t\in X$, $f\in X$. Suppose that this semigroup of composition operators is strongly continuous.
	
	Given this strongly continuous semigroup $(T_t)$, its \textit{infinitesimal generator} $\Gamma_X$ is defined through
	\begin{equation}\label{eq:infinitesimal generator}
		\Gamma_X(f):=\lim\limits_{t\to0}\dfrac{T_t(f)-f}{t}, \quad f\in D_X,
	\end{equation}
	where the limit is evaluated with respect to the norm of $X$. It was proved in \cite[Theorem 3.7]{berksonporta}, \cite[Theorem 1]{Siskakis-Bergman} and \cite[Theorem 1]{Siskakis-Dirichlet} that if $X$ is a Hardy space $H^p$, $p\ge1$, or a Bergman space $A_\alpha^p$, $p\ge1$, $\alpha>-1$, or the Dirichlet space $\mathcal{D}$, respectively, then the domain of $\Gamma_X$ may be written as
	\begin{equation}\label{eq:domain of generator}
		D_X=\left\{f\in X:Gf'\in X\right\},
	\end{equation}
	where $G$ is the infinitesimal generator of the inducing semigroup $(\phi_t)$. This greatly facilitates research and allows us to interpret geometrically the point spectrum of $\Gamma_X$ in subsequent sections.
	
	Furthermore, given a strongly continuous semigroup $(T_t)$, its \textit{type} (also known as \textit{growth bound}) is defined as
	\begin{equation}\label{eq:type}
		\omega_X:=\lim\limits_{t\to+\infty}\dfrac{\log||T_t||_X}{t},
	\end{equation}
	where now $||\cdot||_X$ denotes the norm of the operator with respect to the Banach space $X$. We prefer the notation $\omega_X$ instead of $\omega$ in order to avoid confusion, since the same semigroup may be viewed under the scope of different Banach spaces. In addition, through the type of the semigroup, we can define the \textit{spectral radius} $r_X(T_t)$ of each operator $T_t$, $t\ge0$. Indeed, $r_X(T_t):=e^{\omega_X t}$.
	
	Through the type of a semigroup, we are also able to extract a useful inclusion about the full spectrum of its infinitesimal generator. In general we will use the notation $\sigma_X:=\sigma_X(\Gamma_X)$ for the \textit{point spectrum} of $\Gamma_X$ whenever the semigroup acts on the space $X$, and the notation $\sigma_X^f:=\sigma_X^f(\Gamma_X)$ for the respective \textit{full spectrum}. We know that
	\begin{equation}\label{eq:full spectrum type}
		\sigma_X^f\subseteq\{\lambda\in\C:\mathrm{Re}\lambda\le\omega_X\}.
	\end{equation}
	Of course, since $\sigma_X\subseteq \sigma_X^f$, the same inclusion is also trivially valid for the point spectrum.

	\hiddennumberedsubsection{Conformal Invariants} \label{sub:conformal invariants}
	At this point, we take a moment to talk about two conformally invariant quantities that we need during the course of the article. We start with certain notions of hyperbolic geometry. For a concise presentation of the hyperbolic quantities we utilize, we refer to \cite[Chapter 5]{Booksem}. The \textit{hyperbolic metric} in the unit disk is defined through the formula
	\begin{equation}\label{eq:hyperbolic metric D}
		\lambda_{\D}(z)|dz|=\dfrac{|dz|}{1-|z|^2},
	\end{equation}
	where the function $\lambda_{\D}$ is called the \textit{hyperbolic density} of $\D$. Then, given $z_1,z_2\in\D$, their \textit{hyperbolic distance} $d_{\D}(z_1,z_2)$ in the unit disk is 
	\begin{equation}\label{eq:hyperbolic distance D}
		d_{\D}(z_1,z_2)=\inf\limits_{\gamma}\int\limits_{\gamma}\lambda_{\D}(\zeta)|d\zeta|,
	\end{equation}
	where the infimum is taken over all piecewise $C^1$-smooth curves $\gamma:[0,1]\to\D$ satisfying $\gamma(0)=z_1$ and $\gamma(1)=z_2$. For the purposes of the present work, we are primarily concerned with one instance of the hyperbolic distance. More specifically, it can be proved that
	\begin{equation}\label{eq:hyperbolic distance 0}
		d_{\D}(0,z)=\dfrac{1}{2}\log\dfrac{1+|z|}{1-|z|}, \quad z\in\D.
	\end{equation}
	
	We move on to a second conformal invariant, \textit{Green's function}. Let $\Omega\subsetneq\C\cup\{\infty\}$ be a domain with non-polar boundary (i.e. of positive logarithmic capacity). Then, for each $w\in\Omega$ there exists a unique function $g_\Omega(\cdot,w):\Omega\to(-\infty,+\infty]$ such that
	\begin{enumerate}
		\item[(i)] $g_\Omega(\cdot,w)$ is harmonic on $\Omega\setminus\{w\}$ and bounded outside each neighborhood of $\{w\}$;
		\item[(ii)] $g_\Omega(w,w)=\infty$ and as $z\to w$
		$$g_\Omega(z,w)= \begin{cases}
			\log|z|+O(1),\quad w=\infty\\
			\log|z-w|+O(1), \quad w\ne\infty;
		\end{cases}$$
		\item[(iii)] $g_\Omega(z,w)\to 0$, as $z\to\zeta$, for nearly every $\zeta\in\partial\Omega$, where ``nearly every'' means for all $\zeta\in\partial\Omega\setminus E$ for some $E\subset\partial\Omega$ that has zero logarithmic capacity.
	\end{enumerate}
	This function is called the \textit{Green's function of} $\Omega$ \textit{with pole at} $w$. For background on Green's function and the proofs of results not proven below, we refer to \cite[Section 4.4]{Ransford}. Every Green's function is symmetric in the sense that $g_\Omega(z,w)=g_\Omega(w,z)$, for all $z,w\in\Omega$. Especially in the unit disk, we know the formula
	\begin{equation}\label{eq:green unit disk}
		g_{\D}(0,z)=\log\dfrac{1}{|z|}, \quad \textup{for all }z\in\D.
	\end{equation}
	In addition, Green's function is conformally invariant. In other words, given two domains $\Omega_1,\Omega_2\subsetneq\C\cup\{\infty\}$ with non-polar boundaries and a conformal mapping $f:\Omega_1\to\Omega_2$, we have
	\begin{equation}\label{eq:green invariance}
		g_{\Omega_1}(z,w)=g_{\Omega_2}(f(z),f(w)), \quad\mathrm{for \; all\;}z,w\in\Omega_1.
	\end{equation}
	Combining this property with the known formula for the unit disk, we are able to explicitly compute the Green's function for certain standard domains. We provide an example that will be needed later on.
	\begin{example}[{\cite[p.109]{Ransford}}]\label{ex:green sector}
		Consider the symmetric angular sector $S:=\{w\in\C:|\arg w|<\pi/(2\alpha)\}$, for $\alpha\in[1/2,+\infty)$. Then, for any pair of points $z_1,z_2\in S$, we have
		\begin{equation*}
			g_S(z_1,z_2)=\log\left|\dfrac{z_1^\alpha+\overline{z_2}^\alpha}{z_1^\alpha-z_2^\alpha}\right|.
		\end{equation*}
	\end{example}
	Finally, Green's function satisfies a domain monotonicity property. Given $\Omega_1,\Omega_2$ as above with $\Omega_1\subseteq\Omega_2$, we know that
	\begin{equation}\label{eq:green monotonicity}
		g_{\Omega_1}(z,w)\le g_{\Omega_2}(z,w), \quad\mathrm{for \; all\;}z,w\in\Omega_1.
	\end{equation}

	\hiddennumberedsubsection{Notation}
	To avoid confusion later on, we exhibit here all the pieces of notation that deem necessary in the subsequent sections. Some of them are already present in the Introduction, but we repeat them to aid the interested reader. Firstly, given $\alpha,\beta\in[-\pi,\pi]$ with $\alpha<\beta$, we write $S(\alpha,\beta)=\{w\in\C:\arg w\in(\alpha,\beta)\}$. By $\overline{S(\alpha,\beta)}$ we denote the usual topological closure of the above angular sector. In addition, we will write $S[\alpha,\beta)=\{w\in\C:\arg w\in[\alpha,\beta)\}$ and $S(\alpha,\beta]=\{w\in\C:\arg w\in(\alpha,\beta]\}$. Moving on from angular sectors, we turn our attention towards vertical strips. Given $x>0$, we set $\Sigma_x=\{w\in\C:\mathrm{Re}w\in(-x,0)\}$. Again, by $\overline{\Sigma_x}$ we mean the usual topological closure. In similar fashion as previously, we will need the notation $\widetilde{\Sigma}_x=\{w\in\C:-x\le \mathrm{Re} w<0\}$. Other than that, we use the classical notations $[a,b], (a,b), (a,b]$ and $[a,b)$ for the intervals joining two points $a,b\in\C\cup\{\infty\}$. Finally, by $D(w,r)$ we indicate the Euclidean disk of center $w\in\C$ and radius $r>0$.
	
	\section{\quad The Dirichlet Space}\label{sec:dirichlet}
	We initially study semigroups of composition operators acting on the classical Dirichlet space. We will start with certain fundamental results about such semigroups and the point spectra of their infinitesimal generators, and then move on to a more involved study depending on the type of the holomorphic semigroup inducing the semigroup of composition operators and the Euclidean geometry of the respective Koenigs domain.
	
	\subsection{Basic Properties}
	
	Let $(\phi_t)$ be a non-elliptic semigroup in $\D$, $G$ its infinitesimal generator, $h$ its Koenigs function and $\Omega$ its Koenigs domain. Suppose that $(\phi_t)$ induces $(T_t)$. Recall that by \cite[Theorem 1]{Siskakis-Dirichlet}, its infinitesimal generator $\Gamma_\mathcal{D}$ is defined on the domain $D_\mathcal{D}=\{f\in\mathcal{D}: Gf'\in\mathcal{D}\}$. Our first objective is to extract some general auxiliary results connecting the Koenigs domain $\Omega$ of $(\phi_t)$ with the point spectrum $\sigma_{\mathcal{D}}$ of $\Gamma_\mathcal{D}$. The first result deals with the eigenvalues of the infinitesimal generator and their connection to the Koenigs function of the initial holomorphic semigroup. Its counterpart is known to be true in the Hardy and Bergman spaces; see \cite[p.9]{siskakis_review} and \cite[Theorem 2]{Siskakis-Bergman}, respectively. The proof for the Dirichlet space follows exactly the same steps and relies on solving the differential equation relating the infinitesimal generator $G$ of a holomorphic semigroup with its Koenigs function $h$; see relation \eqref{eq:generator non-elliptic}. For this reason, we opt not to include the proof for the sake of not being redundant.
	
	\begin{proposition}\label{prop:dirichlet spectrum koenigs}
		Let $(\phi_t)$ be a non-elliptic semigroup in $\D$ which induces the semigroup of composition operators $(T_t)$. If $h$ is the Koenigs function of $(\phi_t)$ and $\Gamma_\mathcal{D}$ the infinitesimal generator of $(T_t)$, then $\sigma_{\mathcal{D}}(\Gamma_\mathcal{D})=\{\lambda\in\C:e^{\lambda h}\in\mathcal{D}\}$.
		
	\end{proposition}
	
	\begin{remark}\label{rem:dirichlet zero}
		Since the constant functions belong to the Dirichlet space, we understand that $0$ always belongs to the point spectrum of $\Gamma_\mathcal{D}$, regardless of the initial semigroup $(\phi_t)$.
	\end{remark}
	
	We move on to a lemma that provides a sense of monotonicity for the point spectra.
	
	\begin{lemma}\label{lem:dirichlet spectrum monotonicity}
		Let $(\phi_t)$, $(\widetilde{\phi_t})$ be two non-elliptic semigroups in $\D$ with Koenigs functions $h$, $\widetilde{h}$, and Koenigs domains $\Omega$, $\widetilde{\Omega}$, respectively. Assume that $\Gamma_\mathcal{D}$, $\widetilde{\Gamma_\mathcal{D}}$ are the infinitesimal generators of the respective induced semigroups of composition operators. If $\Omega\subseteq\widetilde{\Omega}$, then $\sigma_{\mathcal{D}}(\Gamma_\mathcal{D})\supseteq\sigma_{\mathcal{D}}(\widetilde{\Gamma_\mathcal{D}})$.
	\end{lemma}
	\begin{proof}
		Clearly $0\in\sigma_{\mathcal{D}}(\Gamma_\mathcal{D})\cap\sigma_{\mathcal{D}}(\widetilde{\Gamma_\mathcal{D}})$. Next, let $\lambda\in\sigma(\widetilde{\Gamma_\mathcal{D}})\setminus\{0\}$. By the previous proposition, $e^{\lambda \widetilde{h}}$ belongs to the Dirichlet space. By the definition of the semi-norm of the Dirichlet space, this translates to 
		\begin{equation}\label{eq:dirichlet spectrum monotonicity 1}
			+\infty>\mathcal{D}\left(e^{\lambda\widetilde{h}}\right)=\int\limits_{\D}\left|(e^{\lambda \widetilde{h}})'(z)\right|^2dA(z)=|\lambda|^2\int\limits_{\D}\left|e^{2\lambda \widetilde{h}(z)}\right|\left|\widetilde{h}'(z)\right|^2dA(z)=\int\limits_{\widetilde{\Omega}}\left|e^{2\lambda w}\right|dA(w),
		\end{equation}
		where in the last equality we executed the change of variables $w=\widetilde{h}(z)$. But by our hypothesis, $\widetilde{\Omega}\supseteq\Omega$, and thus due to the positivity of the integrand, we have
		\begin{equation}\label{eq:dirichlet spectrum monotonicity 2}
			\int\limits_{\widetilde{\Omega}}\left|e^{2\lambda w}\right|dA(w)\ge \int\limits_\Omega\left|e^{2\lambda w}\right|dA(w)=\int\limits_{\D}\left|e^{2\lambda h(z)}\right|\left|h'(z)\right|^2dA(z)=\dfrac{1}{|\lambda|^2}\int\limits_{\D}\left|(e^{\lambda h})'(z)\right|^2dA(z),
		\end{equation}
		where once again we used the change of variables $z=h^{-1}(w)$. Note that both usages of the change of variables are valid since $h,\widetilde{h}$ are univalent as Koenigs functions. Combining relations \eqref{eq:dirichlet spectrum monotonicity 1} and \eqref{eq:dirichlet spectrum monotonicity 2}, we see that $e^{\lambda h}\in\mathcal{D}$ which via Proposition \ref{prop:dirichlet spectrum koenigs} leads to $\lambda\in\sigma_{\mathcal{D}}(\Gamma_\mathcal{D})$. Consequently, $\sigma_{\mathcal{D}}(\widetilde{\Gamma_\mathcal{D}})\subseteq\sigma_{\mathcal{D}}(\Gamma_\mathcal{D})$.
	\end{proof}
	
	Finally, we explicitly compute the point spectrum of $\Gamma_\mathcal{D}$ whenever the Koenigs domain $\Omega$ is a standard domain: either a half-strip, or a strip, or an angular sector $S(a,b)$, $-\pi\le a<b \le\pi$. Note that the convexity of $\Omega$ in the positive direction requires that $a\le 0\le b$ for the sector $S(a,b)$ to be a Koenigs domain.
	
	\begin{lemma}\label{lm:dirichlet standard domain}
		Let $(\phi_t)$ be a non-elliptic semigroup in $\D$ which induces the semigroup of composition operators $(T_t)$. Suppose that $\Omega$ is the Koenigs domain of $(\phi_t)$ and that $\sigma_\mathcal{D}$ is the point spectrum of the infinitesimal generator of $(T_t)$.
		\begin{enumerate}
			\item[\textup{(a)}] If $\Omega$ is a horizontal half-strip, then $\sigma_\mathcal{D}=iS(0,\pi)\cup\{0\}$.
			\item[\textup{(b)}] If $\Omega$ is a horizontal strip, then $\sigma_\mathcal{D}=\{0\}$.
			\item[\textup{(c)}] If $\Omega=S(a,b)$ and $b-a\in(0,\pi)$, then $\sigma_\mathcal{D}=iS(-a,\pi-b)\cup\{0\}$.
			\item[\textup{(d)}] If $\Omega=S(a,b)$ and $b-a\in[\pi,2\pi]$, then $\sigma_\mathcal{D}=\{0\}$.
		\end{enumerate}
		\begin{proof}
			By Remark \ref{rem:dirichlet zero}, $0$ is an eigenvalue of the infinitesimal generator in any case. Let $\lambda\in\sigma_\mathcal{D}\setminus\{0\}$ and set $h$ the Koenigs function of $(\phi_t)$. By Proposition \ref{prop:dirichlet spectrum koenigs}, the inclusion of $\lambda$ in the point spectrum is equivalent to the inclusion of $e^{\lambda h}$ in the Dirichlet space $\mathcal{D}$. We may now distinguish cases.
			
			(a) Since $\Omega$ is convex in the positive direction, the horizontal half-strip necessarily stretches to the right. Then, there exist $x_0,y_1,y_2\in\mathbb{R}$ with $y_1<y_2$ such that $\Omega=\{x+iy\in\C:x>x_0,y_1<y<y_2\}$. Clearly, by definition, $e^{\lambda h}\in\mathcal{D}$ if and only if
			\begin{equation}\label{eq:dirichlet general 1}
				\int\limits_{\D}\left|\left(e^{\lambda h}\right)'(z)\right|^2dA(z)=|\lambda|^2\int\limits_{\D}\left|e^{2\lambda h(z)}\right||h'(z)|^2dA(z)<+\infty.
			\end{equation}
			Given that $h$ is univalent, the change of variables $w=h(z)$ and use of Cartesian coordinates in \eqref{eq:dirichlet general 1} yields
			\begin{equation}\label{eq:dirichlet general 2}
				+\infty>\int\limits_{\Omega}|e^{2\lambda w}|dA(w)=\int\limits_{x_0}^{+\infty}\int\limits_{y_1}^{y_2}e^{2\mathrm{Re}\lambda x}e^{-2\mathrm{Im}\lambda y}dydx.
			\end{equation}
			Evidently, the integral with respect to $y$ is bounded regardless of $\lambda$ and hence combining everything and returning to \eqref{eq:dirichlet general 2}, we get that $\lambda\in\sigma_\mathcal{D}\setminus\{0\}$ if and only if $\int_{x_0}^{+\infty}e^{2\mathrm{Re}\lambda x}dx<+\infty$. But trivially, the last integral is finite if and only if $\mathrm{Re}\lambda<0$, and we have the desired outcome.
			
			(b) Following exactly the same procedure with a change of variables and Cartesian coordinates, we find that $\sigma_\mathcal{D}\setminus\{0\}=\emptyset$.
			
			(c) Executing the same initial steps with the change of variables $w=h(z)$, we understand that $\lambda\in\sigma_\mathcal{D}\setminus\{0\}$ if and only if
			\begin{equation}\label{eq:dirichlet general 3}
				\int\limits_{\Omega} e^{2\mathrm{Re}(\lambda w)} dA(w)<+\infty.
			\end{equation}
			Recall that due to our hypothesis, $\Omega$ may be written as $\{re^{i\theta}:r>0,\theta\in(a,b)\}$. Thus, we will employ polar coordinates. Set $\lambda=|\lambda|e^{i\psi}$. Then, going back to \eqref{eq:dirichlet general 3}, $\lambda\in\sigma_\mathcal{D}\setminus\{0\}$ if and only if
			\begin{equation}\label{eq:dirichlet general 4}
				\int\limits_{a}^{b}\int\limits_{0}^{+\infty}e^{2|\lambda|\cos(\theta+\psi)r}rdrd\theta<+\infty.
			\end{equation}
			The angular sector $S(a,b)=\Omega$ is invariant with respect to scalings, that is $|c|\Omega=\Omega$, for all $c\ne0$. As a result, the finiteness of the latter integral does not depend on the modulus of $\lambda$. On the other hand, since the integral in \eqref{eq:dirichlet general 4} is easily calculated, we see that it is finite if and only if $\cos(\theta+\psi)<0$, for all $\theta\in(a,b)$. This leads to $\theta+\psi\in(\pi/2,3\pi/2)$, for all $\theta\in(a,b)$, which in general means that $\psi\in[\pi/2-a,3\pi/2-b]$. Next, assuming the negativity of $\cos(\theta+\psi)$, we may continue the computations in \eqref{eq:dirichlet general 4} to find that $\lambda\in\sigma_\mathcal{D}\setminus\{0\}$ if and only if
			\begin{equation*}
				+\infty>\dfrac{1}{4|\lambda|^2}\int\limits_{a}^{b}\dfrac{d\theta}{\cos^2(\theta+\psi)}=\dfrac{1}{4|\lambda|^2}\left(\lim\limits_{\theta\to b}\tan(\theta+\psi)-\lim\limits_{\theta\to a}\tan(\theta+\psi)\right).
			\end{equation*}
			Hence, the finiteness of the integral does not allow $\psi=\pi/2-a$ and $\psi=3\pi/2-b$. As a result, we see that $\lambda\in\sigma_\mathcal{D}\setminus\{0\}$ if and only if $\arg\lambda\in(\pi/2-a,3\pi/2-b)$ which implies the desired outcome.
			
			(d) Working as previously, we see that $\lambda\in\sigma_\mathcal{D}\setminus\{0\}$ if and only if $\cos(\theta+\arg\lambda)<0$, for all $\theta\in(a,b)$ which signifies that $\arg\lambda\in[\pi/2-a,3\pi/2-b]$. But now $b-a\in[\pi,2\pi)$ which renders the interval where $\arg\lambda$ belongs empty whenever $b-a\in(\pi,2\pi)$. On the other hand, if $b-a=\pi$, we get $\arg\lambda=\pi/2-a$. But computing the integral again, the tangent in the results tends to $+\infty$ in case $\arg\lambda=\pi/2-a$. As a result, in this case $\sigma_\mathcal{D}\setminus\{0\}=\emptyset$.
		\end{proof}
	\end{lemma}
	
	\begin{remark}\label{rem:dirichlet sectors}
		Since multiplications with a complex non-zero constant preserve the finiteness of the norm in the Dirichlet space, we understand that the preceding lemma remains true in case the Koenigs domain is $w+S(a,b)$, $-\pi\le a<b<\pi$, for some $w\ne0$.
	\end{remark}

	\subsection{Hyperbolic Semigroups}
	
	We are now ready to proceed to the examination of the point spectrum of $\Gamma_\mathcal{D}$ with respect to the Euclidean geometry of $\Omega$. The main tool in our study will be the three handy results we proved above. We start with hyperbolic semigroups. In his articles \cite{betsakos_eig} and \cite{Betsakos_Bergman}, Betsakos set the framework for computing the point spectrum with respect to the inner geometry of the Koenigs domain, for the Hardy and the Bergman spaces, respectively. Our aim is to act similarly for the Dirichlet space. We will realize, later on, that the Dirichlet space presents vast differences compared to the $H^p$ and the $A_\alpha^p$ spaces.
	
	Given a hyperbolic semigroup $(\phi_t)$ with Koenigs function $h$ and Koenigs domain $\Omega$, we know that $\Omega$ is necessarily contained in a horizontal strip. Let $\Sigma=\{w\in\C:a<\mathrm{Im}w<b\}$, where $a,b\in\mathbb{R}$, $a<b$, be the smallest strip containing $\Omega$. It is known that such a strip exists due to the convexity of $\Omega$ in the positive direction and its width depends on the so-called \textit{spectral value} of $(\phi_t)$ (cf. \cite[Theorem 9.4.10]{Booksem}). Then, clearly, the defining function $\psi_\Omega$ of $(\phi_t)$ has domain $(a,b)$.
	
	We start with a first fundamental lemma providing the smallest and largest possible point spectra.
	
	\begin{lemma}\label{lm:hyperbolic basic condition}
		Let $(\phi_t)$ be a hyperbolic semigroup in $\D$ which induces the semigroup of composition operators $(T_t)$. Suppose that $\Omega$ is the Koenigs domain of $(\phi_t)$ and that $\sigma_\mathcal{D}$ is the point spectrum of the infinitesimal generator of $(T_t)$. Then $\{0\}\subseteq\sigma_{\mathcal{D}}\subseteq iS(0,\pi)\cup\{0\}$. In particular, both inclusions are sharp.
	\end{lemma}
	\begin{proof}
		As we mentioned already, $\Omega$ is contained in a horizontal strip. On the other side, by its convexity in the positive direction, $\Omega$ certainly contains a horizontal half-strip stretching to the right. Therefore, a combination of Proposition \ref{prop:dirichlet spectrum koenigs} and Lemma \ref{lm:dirichlet standard domain}(a),(b) yields the desired inclusion and its sharpness. 
	\end{proof}
	
	So if $\lambda\in\sigma_\mathcal{D}\setminus\{0\}$, it is necessary that $\mathrm{Re}\lambda<0$. With this important piece of information in mind, we may prove our first main result, Theorem \ref{thm:intro defining function}.
	
	\begin{proof}[Proof of Theorem \ref{thm:intro defining function}]
		(a) Let $h$ be the Koenigs function of $(\phi_t)$ and fix $\lambda\ne0$. By Proposition \ref{prop:dirichlet spectrum koenigs}, $\lambda\in\sigma_\mathcal{D}\setminus\{0\}$ if and only if $e^{\lambda h}\in\mathcal{D}$ which is equivalent to
		\begin{equation}\label{eq:th1 1}
			\int\limits_{\D}\left|\left(e^{\lambda h}\right)'(z)\right|^2dA(z)=|\lambda|^2\int\limits_{\Omega}|e^{2\lambda w}|dA(w)<+\infty,
		\end{equation}
		where as usual we carry out the change of variables $w=h(z)$. Recall that with the help of the defining function, the Koenigs domain may be parametrized as $\Omega=\{x+iy\in\C:a<y<b, \; x>\psi_\Omega(y)\}$. Therefore, using Cartesian coordinates in \eqref{eq:th1 1}, we see that $\lambda\in\sigma_\mathcal{D}\setminus\{0\}$ if and only if
		\begin{equation}\label{eq:th1 2}
			+\infty>\int\limits_{a}^{b}\int\limits_{\psi_\Omega(y)}^{+\infty}e^{2\mathrm{Re}(\lambda(x+iy))}dxdy=\int\limits_{a}^{b}e^{-2\mathrm{Im}\lambda y}\int\limits_{\psi_\Omega(y)}^{+\infty}e^{2\mathrm{Re}\lambda x}dxdy.
		\end{equation}
		Since $\Omega$ is contained in a horizontal strip, the quantity $e^{-2\mathrm{Im}\lambda y}$ is bounded above and below, and thus does not factor in the finiteness of the latter integral. As a consequence, continuing the calculations in \eqref{eq:th1 2}, we deduce that $\lambda\in\sigma_\mathcal{D}\setminus\{0\}$ if and only if
		\begin{equation}
			+\infty>\int\limits_{a}^{b}\int\limits_{\psi_\Omega(y)}^{+\infty}e^{2\mathrm{Re}\lambda x}dxdy=-\frac{1}{2\mathrm{Re}\lambda}\int\limits_{a}^{b}e^{2\mathrm{Re}\lambda\psi_\Omega(y)}dy,
		\end{equation}
		where the last equality is taking into account the fact that $\mathrm{Re}\lambda<0$ by Lemma \ref{lm:hyperbolic basic condition}. Hence, the desired result follows.
		
		The equivalent condition for the inclusion of $\lambda\ne0$ in the point spectrum, clearly implies that $\sigma_\mathcal{D}\setminus\{0\}$ (recall that $0$ is always an eigenvalue) is invariant with respect to vertical translations. This is because only the real part of $\lambda$ factors into the integral, while the condition is independent of the imaginary part. Therefore, $\sigma_\mathcal{D}\setminus\{0\}$ is either empty, or a vertical strip, or a vertical half-plane. Set
		$$c_\Omega=\inf\left\{c\in(-\infty,0):\int\limits_{a}^{b}e^{2c\psi_\Omega(y)}dy<+\infty\right\}.$$
		
		(b) Suppose, first, that $c_\Omega=0$. Let $\lambda\in\sigma_{\mathcal{D}}\setminus\{0\}$. Then, $\mathrm{Re}\lambda<0$ and due to (a), $\int_{a}^{b}e^{2\mathrm{Re}\lambda \psi_\Omega(y)}dy<+\infty$. Therefore, $c_\Omega\le\mathrm{Re}\lambda<0$. Contradiction! Consequently, $\sigma_\mathcal{D}\setminus\{0\}=\emptyset$.
		
		(c) Suppose $c_\Omega\in(-\infty,0)$. Then, if $\lambda\ne0$ satisfies $\mathrm{Re}\lambda\in(c_\Omega,0)$, the definition of $c_\Omega$ implies $\int_{a}^{b}e^{2\mathrm{Re}\lambda\psi_\Omega(y)}dy<+\infty$, which through (a) yields $\lambda\in\sigma_\mathcal{D}\setminus\{0\}$. On the other hand, if $\mathrm{Re}\lambda<c_\Omega$, the corresponding integral is infinite and hence $\lambda\notin\sigma_\mathcal{D}\setminus\{0\}$ due to (a). Finally, the added condition of the hypothesis shows that if $\mathrm{Re}\lambda=c_\Omega$, $\int_{a}^{b}e^{2\mathrm{Re}\lambda\psi_\Omega(y)}dy=+\infty$ and hence $\lambda\ne\sigma_\mathcal{D}\setminus\{0\}$. The invariance of the point spectrum minus $0$ with respect to vertical translations implies the desired outcome.
		
		(d) Following exactly the same steps as in the previous case, the result follows.
		
		(e) Suppose $c_\Omega=-\infty$. Then due to the infimum, it is evident that $\int_{a}^{b}e^{2\mathrm{Re}\lambda\psi_\Omega(y)}dy<+\infty$, for all $\lambda\ne0$ with $\mathrm{Re}\lambda<0$. Therefore, by (a), the point spectrum $\sigma_\mathcal{D}$ is the whole left half-plane plus the eigenvalue $0$.
	\end{proof}
	
	From our theorem, we understand that the whole shape of the boundary of $\Omega$ plays a role in determining the point spectrum $\sigma_\mathcal{D}$. On the other side, the width $b-a$ is not present in the description of $\sigma_\mathcal{D}$ and neither is the width of any strip contained inside $\Omega$. In particular, a combination of Proposition \ref{prop:dirichlet spectrum koenigs} and Lemma \ref{lm:dirichlet standard domain}(a) shows that if $\Omega$ contains any horizontal strip, then the point spectrum is trivial. As a result, comparing with the results of Betsakos, composition operators induced by hyperbolic semigroups and acting on the Dirichlet space, produce vast dissimilarities to those acting on Bergman and Hardy spaces. Finally, a direct corollary of Theorem \ref{thm:intro defining function} is that $\sigma_\mathcal{D}$ is starlike with respect to $0$ and hence connected. By starlike with respect to $0$, we mean that for any $\lambda\in\sigma_\mathcal{D}\setminus\{0\}$, the whole line segment $[0,\lambda]$ is contained in the point spectrum. We will see that this connectedness fails, in general, in the case of composition operators induced by parabolic semigroups.
	
	Nevertheless, as mentioned in the Introduction, the defining function of a semigroup can be significantly difficult to manipulate. For this reason, we aim to discover a more transparent connection of the point spectrum with the geometry of the Koenigs domain, even if it leads to a partial result. For each $x\in\mathbb{R}$, consider $\Omega_x$ to be the set of the imaginary parts of all points in $\Omega$ with real part $x$, i.e.
	$$\Omega_x:=\{y\in\mathbb{R}:x+iy\in\Omega\}, \quad x\in\mathbb{R}.$$
	Since $\Omega$ is simply connected and convex in the positive direction, each set $\Omega_x$ consists of countably many intervals, it is always a subset of $(a,b)$ due to the strip $\Sigma$, and the family of sets $\{\Omega_x\}$ is increasing (not necessarily strictly). Therefore, given $x\in\mathbb{R}$, we may write 
	$$\Omega_x=\bigcup\limits_{j\in J_x}I_j,$$
	where $J_x$ is a set of countably many indices depending on $x$ and $\{I_j\}$ is a family of pairwise disjoint bounded open intervals in $\mathbb{R}$. Set $\ell_\Omega(x):=\textup{length}(\Omega_x)=\sum_{j\in J_x}\textup{length}(I_j)$. If for some $x_0\in\mathbb{R}$ (and thus for all $x\le x_0$) we have $\Omega_{x_0}=\emptyset$, we just define $\ell_\Omega(x_0)=0$. Through this well-defined process, we see that $\ell_\Omega(x)$ is basically the length of the intersection of $\Omega$ with the vertical line $\{\mathrm{Re}z=x\}$. Evidently, $\ell_\Omega(x)$ is a (not necessarily strictly) increasing and upper semi-continuous function of $x$ and the minimality of $\Sigma$ means that $\lim_{x\to+\infty}\ell_\Omega(x)=b-a$. On top of that, the limit $\lim_{x\to-\infty}\ell_\Omega(x)$ always exists in $[0,b-a]$.
	
	Having mentioned a part of the required exposition, we may now proceed to a fundamental lemma which will be the main tool in better understanding point spectra in this setting. Even though its proof follows easy basic steps, this result will be crucial for our examination.
	
	\begin{lemma}\label{lm:hyperbolic x parametrization}
		Let $(\phi_t)$ be a hyperbolic semigroup in $\D$ which induces the semigroup of composition operators $(T_t)$. Suppose that $\Omega$ is the Koenigs domain of $(\phi_t)$ and that $\sigma_\mathcal{D}$ is the point spectrum of the infinitesimal generator of $(T_t)$. Then $\lambda\in\sigma_{\mathcal{D}}\setminus\{0\}$ if and only if
		\begin{equation}
			\int\limits_{-\infty}^{0}e^{2\mathrm{Re}\lambda x}\ell_\Omega(x)dx<+\infty.
		\end{equation}
	\end{lemma}
	\begin{proof}
		Working exactly as in previous proofs, we have that $\lambda\in\sigma_\mathcal{D}\setminus\{0\}$ if and only if
		\begin{equation}\label{eq:basic lemma dirichlet 1}
			\int\limits_{\Omega}e^{2\mathrm{Re}\lambda\mathrm{Re}w}e^{-2\mathrm{Im}\lambda\mathrm{Im}w}dA(w)<+\infty.
		\end{equation}
		But due to our previous discussion, $\Omega$ may be parametrized as $\Omega=\{x+iy\in\C:x\in\mathbb{R}, \;y\in\Omega_x\}$. Thus, working on \eqref{eq:basic lemma dirichlet 1}, $\lambda\in\sigma_\mathcal{D}\setminus\{0\}$ if and only if
		\begin{equation}\label{eq:basic lemma dirichlet 2}
			\int\limits_{-\infty}^{+\infty}e^{2\mathrm{Re}\lambda x}\int\limits_{\Omega_x}e^{-2\mathrm{Im}\lambda y}dydx<+\infty.
		\end{equation}
		However, since $(\phi_t)$ is hyperbolic, $\Omega$ is contained inside a horizontal strip. Therefore, for the fixed $\lambda$, the quantity $e^{-2\mathrm{Im}\lambda y}$ is uniformly bounded in $\Omega$. Returning to \eqref{eq:basic lemma dirichlet 2}, this signifies that $\lambda$ is in the point spectrum if and only if
		\begin{equation}\label{eq:basic lemma dirichlet 3}
			\int\limits_{-\infty}^{+\infty}e^{2\mathrm{Re}\lambda x}\int\limits_{\Omega_x}dydx=\int\limits_{-\infty}^{+\infty}e^{2\mathrm{Re}\lambda x}\ell_\Omega(x)dx<+\infty.
		\end{equation}
		Recall that since there exists a minimal strip containing the convex in the positive direction simply connected domain $\Omega$, there exists $L\in(0,+\infty)$ such that $\ell_\Omega(x)\to L$, as $x\to+\infty$. In addition, $\ell_\Omega(x)$ is increasing, so $\ell_\Omega(x)\le L$, for all $x\in\mathbb{R}$. Therefore,
		\begin{equation}\label{eq:basic lemma dirichlet 4}
			\int\limits_{0}^{+\infty}e^{2\mathrm{Re}\lambda x}\ell_\Omega(x)dx \le L\int\limits_{0}^{+\infty}e^{2\mathrm{Re}\lambda x}dx.
		\end{equation}
		But by Lemma \ref{lm:hyperbolic basic condition}, $\mathrm{Re}\lambda<0$ and hence \eqref{eq:basic lemma dirichlet 4} implies that $\int_{0}^{+\infty}e^{2\mathrm{Re}\lambda x}\ell_\Omega(x)dx<+\infty$ in any case.  Combining \eqref{eq:basic lemma dirichlet 3} with \eqref{eq:basic lemma dirichlet 4}, we understand that $\lambda\ne0$ belongs to the point spectrum of the infinitesimal generator if and only if $\int_{-\infty}^{0}e^{2\mathrm{Re}\lambda x}\ell_\Omega(x)dx<+\infty$, and we have the desired result.

	\end{proof}
	
	We continue with some more notation before proving the main result of the section. Let $(\phi_t)$ be a hyperbolic semigroup in $\D$ and maintain the usual notation $h, \Omega, (T_t), \Gamma_\mathcal{D}, \ell_\Omega(x)$. Write
	$$\delta_\Omega:=\lim\limits_{x\to-\infty}\dfrac{\log\ell_\Omega(x)}{2x},$$
	a priori assuming that the limit exists, where we allow $\log\ell_\Omega(x)=-\infty$ in case $\ell_\Omega(x)=0$ for some $x\in\mathbb{R}$. By construction, it is straightforward that $0\le\delta_\Omega\le+\infty$, for any Koenigs domain $\Omega$ of a hyperbolic semigroup. In simple terms, the number $\delta_\Omega$ is a measure of how fast the boundary $\partial\Omega$ diverges from the boundary of the smallest horizontal strip containing $\Omega$, as $x\to-\infty$.
	
	To demonstrate the necessity of the assumption that the limit exists, we now provide an example where the corresponding limit inferior and limit superior are actually different.
	
	\begin{example}\label{ex:limit does not exist}
		Consider the sequences $\{a_n\}, \{b_n\}\subset(0,1)$ with $a_n=e^{-2^n}$ and $b_n=e^{-2^{n+1}}+e^{-3^n}$. For each $n\in\mathbb{N}$ set $x_n=-2^n+ia_n$ and $y_n=-2^n-\frac{1}{n^2}+ib_n$. By construction, we have $\mathrm{Re}x_n>\mathrm{Re}y_n>\mathrm{Re}x_{n+1}$ and $\mathrm{Im}x_n>\mathrm{Im}y_n>\mathrm{Im}x_{n+1}$, for every $n\in\mathbb{N}$. Consider the union of rectilinear segments
		$$\gamma=\bigcup\limits_{n=1}^{+\infty}\left([x_n,y_n]\cup[y_n,x_{n+1}]\right)$$
		and let $\Omega_\gamma$ denote the simply connected domain bounded by $\gamma$, the real axis and the vertical segment $[-2,-2+ie^{-2}]$. Finally, let $\Omega:=\Omega_\gamma\cup\{x+iy\in\C:x\ge-2, \; 0<y<e^{-2}\}$. By construction, $\Omega$ is a simply connected domain that is convex in the positive direction and contained in a horizontal strip. Therefore, it is the Koenigs domain of some hyperbolic semigroup. Then
		$$\liminf\limits_{x\to-\infty}\dfrac{\log\ell_\Omega(x)}{2x}\le\lim\limits_{n\to+\infty}\dfrac{\log\ell_\Omega(-2^n)}{-2\cdot2^n}=\lim\limits_{n\to+\infty}\dfrac{\log a_n}{-2^{n+1}}=\lim\limits_{n\to+\infty}\dfrac{\log e^{-2^n}}{-2^{n+1}}=\dfrac{1}{2}.$$
		On the other hand,
		$$\limsup\limits_{x\to-\infty}\dfrac{\log\ell_\Omega(x)}{2x}\ge\lim\limits_{n\to+\infty}\dfrac{\log\ell_\Omega\left(-2^n-\frac{1}{n^2}\right)}{-2\cdot 2^n-\frac{2}{n^2}}=\lim\limits_{n\to+\infty}\dfrac{\log b_n}{-2^{n+1}-\frac{2}{n^2}}=\lim\limits_{n\to+\infty}\dfrac{\log(e^{-2^{n+1}}+e^{-3^n})}{-2^{n+1}-\frac{2}{n^2}}=1.$$
		As a result, clearly $\delta_\Omega$ does not exist.
	\end{example}
	
	Other than that, we are going to need one more number depending on the Euclidean geometry of $\Omega$. It is clear that $\Omega$ itself has infinite area since it always contains a horizontal half-strip. We may separate $\Omega$ into the two sets $\Omega^+$ and $\Omega^-$, where $\Omega^+:=\{w\in\Omega:\mathrm{Re}w\ge0\}$ and $\Omega^-:=\{w\in\Omega:\mathrm{Re}w< 0\}$. Again, the area of $\Omega^+$ is infinite. On the other hand, the area of $\Omega^-$ might be finite. We write $W(\Omega):=\mathrm{area}(\Omega^-)\in[0,+\infty]$ and we may see that
	$$W(\Omega)=\int\limits_{-\infty}^{0}\ell_\Omega(x)dx.$$
	Again, the number $W(\Omega)$ is an indicator of how fast $\partial\Omega$ moves away from the boundary of the smallest horizontal strip containing $\Omega$, as $x\to-\infty$, and depends solely on the Euclidean geometry of $\Omega$. The values of the two numbers $\delta_\Omega$ and $W(\Omega)$ will play a principal role in determining the point spectrum of $\Gamma_\mathcal{D}$, as demonstrated below. 
	
	\begin{proof}[Proof of Theorem \ref{thm:intro hyperbolic dirichlet}]
		(a) Certainly $0\in\sigma_{\mathcal{D}}$. Suppose that $\lambda\in\sigma_{\mathcal{D}}\setminus\{0\}$. By Lemma \ref{lm:hyperbolic basic condition}, we have $\mathrm{Re}\lambda<0$ and hence $e^{2\mathrm{Re}\lambda x}>1$, for all $x\in(-\infty,0)$. Therefore
		$$\int\limits_{-\infty}^{0}e^{2\mathrm{Re}\lambda x}\ell_\Omega(x)dx\ge\int\limits_{-\infty}^{0}\ell_\Omega(x)=W(\Omega)=+\infty.$$
		However, due to Lemma \ref{lm:hyperbolic x parametrization}, we obtain $\lambda\notin\sigma_{\mathcal{D}}$. Contradiction! This leads to $\sigma_{\mathcal{D}}\setminus\{0\}=\emptyset$ and the desired outcome.
		
		(b) Because of Lemma \ref{lm:hyperbolic basic condition}, it suffices to prove that $iS(0,\pi)\subseteq\sigma_{\mathcal{D}}$. Towards this goal, let $\lambda\in\C$ with $\mathrm{Re}\lambda<0$. Pick $M>0$ such that $-M<\mathrm{Re}\lambda$ or equivalently $M+\mathrm{Re}\lambda>0$. Due to our hypothesis, for this $M$ we may find some $x_0\in(-\infty,0)$ such that $\log\ell_\Omega(x)/(2x)>M$, for all $x\in(-\infty,x_0)$. Elementary rearrangements show that $\ell_\Omega(x)<e^{2Mx}$, for all $x\in(-\infty,x_0)$. Then
		\begin{eqnarray*}
			\int\limits_{-\infty}^{0}e^{2\mathrm{Re}\lambda x}\ell_\Omega(x)dx &=& \int\limits_{-\infty}^{x_0}e^{2\mathrm{Re}\lambda x}\ell_\Omega(x)dx+\int\limits_{x_0}^{0}e^{2\mathrm{Re}\lambda x}\ell_\Omega(x)dx \\
			&\le&\int\limits_{-\infty}^{x_0}e^{2(\mathrm{Re}\lambda+M)x}dx+\int\limits_{x_0}^{0}e^{2\mathrm{Re}\lambda x}\ell_\Omega(x)dx\\
			&<&+\infty,
		\end{eqnarray*}
		since on one hand, $2(\mathrm{Re}\lambda+M)>0$ and on the other hand, the integral from $x_0$ to $0$ is clearly finite. By Lemma \ref{lm:hyperbolic x parametrization} we conclude that $\lambda\in\sigma_{\mathcal{D}}$. The arbitrariness in the choice of $\lambda\in iS(0,\pi)$ implies the desired inclusion.
		
		(c) Due to our assumptions and combining Lemma \ref{lm:hyperbolic x parametrization} with the vertical invariance of $\sigma_\mathcal{D}$, we deduce that $\{w\in\C:\mathrm{Re}w=-\delta_\Omega\}\subseteq\sigma_{\mathcal{D}}$. The shape of the point spectrum imposed by Theorem \ref{thm:intro defining function} shows that $\widetilde{\Sigma}_{\delta_\Omega}\subseteq\sigma_{\mathcal{D}}$. Clearly $0\in\sigma_{\mathcal{D}}$ as well. For the reverse inclusion, assume that $\lambda\in\C$ with $\mathrm{Re}\lambda<-\delta_\Omega$. We are going to prove that $\lambda\notin\sigma_{\mathcal{D}}$. Fix $\epsilon\in(0,-\mathrm{Re}\lambda-\delta_\Omega)$. Due to the assumption on the limit, for this $\epsilon$ we may find $x_1\in(-\infty,0)$ such that $\log\ell_\Omega(x)/(2x)<\delta_\Omega+\epsilon$, for all $x\in(-\infty,x_1)$. Quick calculations yield $\ell_\Omega(x)>e^{2(\delta_\Omega+\epsilon)x}$, for all $x\in(-\infty,x_1)$. Therefore,
		\begin{equation*}
			\int\limits_{-\infty}^{0}e^{2\mathrm{Re}\lambda x}\ell_\Omega(x)dx\ge\int\limits_{-\infty}^{x_1}e^{2(\mathrm{Re}\lambda+\delta_\Omega+\epsilon)x}dx=+\infty,
		\end{equation*}
		since by construction $2(\mathrm{Re}\lambda+\delta_\Omega+\epsilon)<0$. Returning to Lemma \ref{lm:hyperbolic x parametrization} we understand that $\lambda\notin\sigma_{\mathcal{D}}$. Combining with the fact that $\sigma_{\mathcal{D}}\subseteq iS(0,\pi)\cup\{0\}$, we get the required equality for the point spectrum.
		
		(d) Working identically as in the previous case, we may reach the point where $\sigma_{\mathcal{D}}\subseteq \widetilde{\Sigma}_{\delta_\Omega}\cup\{0\}$. Due to the integral condition in our hypothesis, we further get that $\sigma_{\mathcal{D}}\subseteq \Sigma_{\delta_\Omega}\cup\{0\}$. Reversely, let $\lambda\in\C$ with $\mathrm{Re}\lambda\in(-\delta_\Omega,0)$. Fix $\epsilon\in(0,\mathrm{Re}\lambda+\delta_\Omega)$. For this $\epsilon$ we may find $x_2\in(-\infty,0)$ such that $\log\ell_\Omega(x)/(2x)>\delta_\Omega-\epsilon$, for all $x\in(-\infty,x_2)$, due to the existence of the limit. Again, we can rearrange to obtain $\ell_\Omega(x)<e^{2(\delta_\Omega-\epsilon)x}$, for all $x\in(-\infty,x_2)$. Then
		\begin{equation*}
			\int\limits_{-\infty}^{0}e^{2\mathrm{Re}\lambda x}\ell_\Omega(x)dx\le\int\limits_{x_2}^{0}e^{2\mathrm{Re}\lambda x}\ell_\Omega(x)dx+\int\limits_{-\infty}^{x_2}e^{2(\mathrm{Re}\lambda+\delta_\Omega-\epsilon)x}dx<+\infty,
		\end{equation*}
		since $2(\mathrm{Re}\lambda+\delta_\Omega-\epsilon)>0$, while the integral from $x_2$ to $0$ is evidently finite. Returning to Lemma \ref{lm:hyperbolic x parametrization}, we see that $\lambda\in\sigma_{\mathcal{D}}$. Combining with the fact that $0$ is always an eigenvalue, we deduce that $\Sigma_{\delta_\Omega}\cup\{0\}\subseteq\sigma_{\mathcal{D}}$ which produces the equality.
		
		(e) By Lemma \ref{lm:hyperbolic basic condition} it suffices to show that any $\lambda\in iS(0,\pi)$ is not an eigenvalue. Fix such a $\lambda$ with $\mathrm{Re}\lambda<0$ and find $\epsilon\in(0,-\mathrm{Re}\lambda)$. For this $\epsilon$, our assumption on the limit implies the existence of some $x_3\in(-\infty,0)$ so that $\log\ell_\Omega(x)/(2x)<\epsilon$, for all $x\in(-\infty,x_3)$. Similarly to before, this means that $\ell_\Omega(x)>e^{2\epsilon x}$, for all $x\in(-\infty,x_3)$. Consequently,
		\begin{equation*}
			\int\limits_{-\infty}^{0}e^{2\mathrm{Re}\lambda x}\ell_\Omega(x)dx\ge\int\limits_{-\infty}^{x_3}e^{2(\mathrm{Re}\lambda+\epsilon)x}dx=+\infty,
		\end{equation*}
		because $2(\mathrm{Re}\lambda+\epsilon)<0$. Once again, Lemma \ref{lm:hyperbolic x parametrization} yields the desired outcome.
	\end{proof}
	
	As mentioned in the Introduction, in case the limit $\delta_\Omega$ does not exist, we may obtain helpful inclusions with the use of the corresponding limit inferior and limit superior. One of the reasons Theorem \ref{thm:intro hyperbolic dirichlet} fails to provide a complete description is the fact that $\ell_\Omega(x)$ does not fully characterize the semigroup $(\phi_t)$ or the Koenigs domain $\Omega$. As a matter of fact, there can be a variety of Koenigs domains (for instance horizontal strips with horizontal slits) that have the exact same length $\ell_\Omega(x)$, for all $x\in\mathbb{R}$. On the contrary, as demonstrated by Theorem \ref{thm:intro defining function} we are in need of a firmer grasp on the geometry of $\Omega$.
	
	Our theorems on the point spectrum in the Dirichlet space allow us to state certain easy but useful corollaries with respect to the dynamics of the hyperbolic holomorphic semigroup.
	
	\begin{corollary}\label{cor:hyperbolic dirichlet}
		Let $(\phi_t)$ be a hyperbolic semigroup in $\D$ which induces the semigroup of composition operators $(T_t)$. Suppose that $\sigma_\mathcal{D}$ is the point spectrum of the infinitesimal generator of $(T_t)$. Then:
		\begin{enumerate}
			\item[\textup{(a)}] If $(\phi_t)$ is a group, $\sigma_{\mathcal{D}}=\{0\}$.
			\item[\textup{(b)}] If $(\phi_t)$ has a repelling fixed point, $\sigma_{\mathcal{D}}=\{0\}$.
			\item[\textup{(c)}] If $(\phi_t)$ has no boundary fixed points other than its Denjoy--Wolff point, $\sigma_{\mathcal{D}}=iS(0,\pi)\cup\{0\}$.
		\end{enumerate}
	\end{corollary}
	\begin{proof}
		(a) If $(\phi_t)$ is a group with Koenigs domain $\Omega$, we know by \cite[Proposition 9.3.12]{Booksem} that $\Omega$ is exactly a horizontal strip. Therefore, Lemma \ref{lm:dirichlet standard domain}(b) provides the desired point spectrum.
		
		(b) Suppose that $(\phi_t)$ has a repelling fixed point $\sigma$ and that $\Omega$ is its Koenigs domain. By \cite[Theorem 13.5.5]{Booksem} we know that $\Omega$ contains a horizontal strip. Then, $\ell_\Omega(x)$ remains positive and bounded by below for all $x\in\mathbb{R}$. We easily see that $W(\Omega)=+\infty$ and hence by Theorem \ref{thm:intro hyperbolic dirichlet}(a), we calculate the point spectrum.
		
		(c) Again set $\Omega$ the Koenigs domain of $(\phi_t)$ and $h$ its Koenigs function. We claim that our hypothesis signifies that $\inf_{w\in\Omega}\mathrm{Re}w>-\infty$. Suppose, conversely, that $\inf_{w\in\Omega}\mathrm{Re}w=-\infty$ and thus there exists a sequence $\{w_n\}\subset\Omega$ such that the sequence $\{\mathrm{Re}w_n\}$ strictly decreases to $-\infty$, as $n\to+\infty$. As a result, there exists at least one prime end of $\Omega$, different from the one corresponding to the Denjoy--Wolff point of $(\phi_t)$, containing $\infty$ in its impression. But $(\phi_t)$ is hyperbolic and hence $\Omega$ is a convex in the positive direction simply connected domain contained in a horizontal strip. Straightforward geometric considerations show that the existence of such a prime end implies that necessarily there exists at least one prime end $\xi$ of $\Omega$, different from the one corresponding to the Denjoy--Wolff point, such that its impression is exactly $\infty$. By Carath\'{e}odory's Theorem, this prime end $\xi$ corresponds through $h^{-1}$ to a point $\sigma\in\partial\D$. But our construction also dictates that $\lim_{z\to\sigma}\mathrm{Re}h(z)=-\infty$. By \cite[Proposition 13.6.2]{Booksem} we get that $\sigma$ is a boundary fixed point for $(\phi_t)$, different from the Denjoy--Wolff point. Due to our hypothesis, we have reached a contradiction. Therefore, indeed $\inf_{w\in\Omega}\mathrm{Re}w>-\infty$ and $\Omega$ is contained in a horizontal half-strip that stretches to the right. For this reason, we may find some $x_0\in\mathbb{R}$ such that $\ell_\Omega(x)=0$, for all $x<x_0$. Trivially $W(\Omega)<+\infty$ and $\delta_\Omega=+\infty$. By Theorem \ref{thm:intro hyperbolic dirichlet}(b) we get the desired point spectrum for $\Gamma_\mathcal{D}$.
	\end{proof}

	Our last corollary provides several non-trivial examples where the point spectrum of $\Gamma_\mathcal{D}$ may be explicitly computed. Since our results have already provided examples where $\sigma_{\mathcal{D}}$ is either $\{0\}$ or $iS(0,\pi)\cup\{0\}$, we only lack examples where this spectrum is a strip. To end the subsection, we will construct an example of a semigroup (or rather its Koenigs domain) which will give $\sigma_{\mathcal{D}}=\Sigma_{\epsilon}\cup\{0\}$, for any initially prescribed $\epsilon\in(0,+\infty)$. Tweaking slightly this construction, we will also provide an example where the spectrum is equal to $\widetilde{\Sigma}_\epsilon\cup\{0\}$.
	
	\begin{example}
		Fix $\epsilon\in(0,+\infty)$. Consider the convex in the positive direction simply connected domain 
		$$\Omega:=\{x+iy\in\C:x<0,\; |y|<e^{2\epsilon x}\}\cup\{x+iy\in\C:x\ge0,\; |y|<1\}.$$
		Because of its geometry, $\Omega$ is the Koenigs domain of a semigroup $(\phi_t)$. It is evident that $\Omega$ is contained in a horizontal strip and therefore $(\phi_t)$ is hyperbolic. Let $(T_t)$ be the induced semigroup of composition operators and $\Gamma_\mathcal{D}$ its infinitesimal generator. Clearly, $\ell_\Omega(x)=2$ for $x\ge0$, while $\ell_\Omega(x)=2e^{2\epsilon x}$ for $x<0$. It is easy to calculate through the integral in the definition that $W(\Omega)=1/\epsilon<+\infty$. In addition, it is readily verified that $\delta_\Omega=\epsilon$. Finally, 
		$$\int\limits_{-\infty}^{0}e^{-2\epsilon x}\ell_\Omega(x)dx=2\int\limits_{-\infty}^{0}e^{-2\epsilon x}e^{2\epsilon x}dx=+\infty.$$ 
		By Theorem \ref{thm:intro hyperbolic dirichlet}(d) we deduce that $\sigma_{\mathcal{D}}=\Sigma_\epsilon\cup\{0\}$.
	\end{example}
	
	\begin{example}
		For our second example, again fix $\epsilon\in(0,+\infty)$ and set
		$$\Omega:=\left\{x+iy\in\C:x<-1, \; |y|<\dfrac{e^{2\epsilon x}}{x^2}\right\}\cup\left\{x+iy\in\C:x\ge-1, \; |y|<\dfrac{1}{e^{2\epsilon}}\right\}.$$
		As previously, $\Omega$ is a convex in the positive direction simply connected domain contained in a horizontal strip and as such it corresponds to a hyperbolic semigroup $(\phi_t)$. Set $(T_t)$ the induced semigroup of composition operators and $\Gamma_\mathcal{D}$ its respective infinitesimal generator. Obviously, the current Koenigs domain is a subset of the Koenigs domain of the previous example. Therefore, by Lemma \ref{lem:dirichlet spectrum monotonicity} we already expect $\sigma_{\mathcal{D}}\supseteq\Sigma_\epsilon\cup\{0\}$. We also intuitively comprehend that this time $\partial\Omega$ moves away from the boundary of the smallest horizontal strip containing $\Omega$ more rapidly than in the previous example. As a matter of fact, this difference will lead to a larger spectrum. Since we have a smaller Koenigs domain than before, we immediately get $W(\Omega)<+\infty$. Furthermore, one may find $\delta_\Omega=\epsilon$ since $\ell_\Omega(x)=2e^{2\epsilon x}/x^2$, for $x<-1$. Finally, 
		$$\int\limits_{-\infty}^{-1}e^{-2\epsilon x}\ell_\Omega(x)dx=2\int\limits_{-\infty}^{-1}e^{-2\epsilon x}\dfrac{e^{2\epsilon x}}{x^2}dx=2\int\limits_{-\infty}^{-1}\dfrac{dx}{x^2}=2<+\infty,$$
		and the corresponding integral from $-1$ to $0$ is obviously finite as well. By Theorem \ref{thm:intro hyperbolic dirichlet}(c) we obtain $\sigma_{\mathcal{D}}=\widetilde{\Sigma}_\epsilon\cup\{0\}$.
	\end{example}

	\subsection{Parabolic Semigroups}
	Having concluded our study on the point spectrum of the infinitesimal generators induced by hyperbolic semigroups, we proceed to parabolic semigroups. We commence with certain general results concerning all parabolic semigroups. Then, we will proceed to a separation between those of positive hyperbolic step and those of zero hyperbolic step. Recall that the Koenigs domain of a parabolic semigroup is not contained in any horizontal strip. 
	We start with a basic lemma that gives a first idea of the smallest possible and the largest possible point spectra that can arise from parabolic semigroups. Its proof is easy and can be derived from arguments which are related to hyperbolic semigroups. Later on we will see that the following result is indeed sharp.
	
	\begin{lemma}\label{lm:phs general}
		Let $(\phi_t)$ be a parabolic semigroup in $\D$ which induces the semigroup of composition operators $(T_t)$. Suppose that $\sigma_\mathcal{D}$ is the point spectrum of the infinitesimal generator of $(T_t)$. Then
		$$\{0\}\subseteq\sigma_\mathcal{D}\subseteq iS(0,\pi)\cup\{0\}.$$
	\end{lemma}
	\begin{proof}
		By Remark \ref{rem:dirichlet zero} it is clear that $0\in\sigma_\mathcal{D}$. On the other side, the Koenigs domain $\Omega$ of $(\phi_t)$ is simply connected and convex in the positive direction. Therefore, it contains a horizontal half-strip stretching to the right. Thus, a combination of Proposition \ref{prop:dirichlet spectrum koenigs} with Lemma \ref{lm:dirichlet standard domain}(a) provides the desired right-hand side inclusion.
	\end{proof}
	
	We now consider semigroups of composition operators induced by parabolic semigroups of positive hyperbolic step. As we already mentioned in the Introduction, the main geometric properties that will be examined with regard to the point spectrum in this case are the inner and outer arguments. These numbers are inextricably linked with the intrinsic Euclidean geometry of the Koenigs domain $\Omega$ of $(\phi_t)$. For this reason and to give emphasis on the relation to the Koenigs domain, we denote them by $\theta_\Omega$ and $\Theta_\Omega$. We proceed to the main theorem of this subsection, Theorem \ref{thm:intro phs dirichlet}. Recall that its statement refers to the case when the Koenigs domain is contained in the usual upper half-plane. If $\Omega$ was contained in any other horizontal half-plane, the same statement holds, but slightly modified in the obvious sense.

	\begin{proof}[Proof of Theorem \ref{thm:intro phs dirichlet}]
		(a) Due to the outer argument, for every $\theta\in(0,\pi)$, we may find some $w_\theta\in\C$ such that $\Omega\subset w_\theta+S(0,\theta)$. Each set $w_\theta+S(0,\theta)$ is simply connected and convex in the positive direction and hence corresponds to some non-elliptic semigroup. A combination of Proposition \ref{prop:dirichlet spectrum koenigs}, Lemma \ref{lm:dirichlet standard domain}(c) and Remark \ref{rem:dirichlet sectors} shows that $\sigma_\mathcal{D}\supseteq iS(0,\pi-\theta)\cup\{0\}$, for all $\theta\in(0,\pi)$. Ergo 
		$$\sigma_\mathcal{D}\supseteq \{0\} \cup \left(i\bigcup_{\theta\in(0,\pi)}S(0,\pi-\theta)\right)$$ which means that $\sigma_\mathcal{D}\supseteq iS(0,\pi)\cup\{0\}$. By Lemma \ref{lm:phs general} we acquire the equality.
		
		(b) Exactly as in the previous case, the outer argument leads to $\sigma_\mathcal{D}\supseteq iS(0,\pi-\theta_\Omega)\cup\{0\}$. On the other side, due to the inner argument, we have that for each $\theta\in(0,\theta_\Omega)$ there exists some $w_\theta\in\C$ such that $w_\theta+S(0,\theta)\subset\Omega$. Again, combining Proposition \ref{prop:dirichlet spectrum koenigs}, Lemma \ref{lm:dirichlet standard domain}(c) and Remark \ref{rem:dirichlet sectors} yields $\sigma_\mathcal{D}\subseteq iS(0,\pi-\theta)\cup\{0\}$, for all $\theta\in(0,\theta_\Omega)$. Therefore, the point spectrum is a subset of the intersection of all these sets which signifies that $\sigma_\mathcal{D}\subseteq i\overline{S(0,\pi-\theta_\Omega)}\cup\{0\}$. Nevertheless, by Lemma \ref{lm:phs general} we know that a non-zero point on the imaginary axis cannot belong to the point spectrum. For this reason, $\sigma_\mathcal{D}\subseteq iS(0,\pi-\theta_\Omega]\cup\{0\}$.
		
		(c) Following identical steps with the inner argument as in the preceding case, we reach the point $\sigma_\mathcal{D}\subseteq iS(0,\pi-\theta)$, for all $\theta\in(0,\pi)$. As a consequence, $\sigma_\mathcal{D}\subseteq\{it:t\ge0\}$. Once more, Lemma \ref{lm:phs general} certifies that $\sigma_\mathcal{D}=\{0\}$.
		
		(d) An execution of arguments similar to those in the previous case with the inner and outer arguments provides the desired inclusion. We omit its proof for the sake of avoiding redundancy.
	\end{proof}
	
	Theorem \ref{thm:intro phs dirichlet} leads to a straightforward corollary relating the point spectrum of the infinitesimal generator to the dynamical properties of the inducing semigroup.
	
	\begin{corollary}\label{cor:phs dirichlet}
		Let $(\phi_t)$ be a parabolic semigroup of positive hyperbolic step in $\D$ which induces the semigroup of composition operators $(T_t)$. Suppose that $\sigma_\mathcal{D}$ is the point spectrum of the infinitesimal generator of $(T_t)$. Then:
		\begin{enumerate}
			\item[\textup{(a)}] If $(\phi_t)$ is a group, $\sigma_\mathcal{D}=\{0\}$.
			\item[\textup{(b)}] If $(\phi_t)$ has a repelling fixed point, $\sigma_\mathcal{D}=\{0\}$.
			\item[\textup{(c)}] If $(\phi_t)$ is of finite shift, $\sigma_\mathcal{D}=\{0\}$. 
		\end{enumerate}
	\end{corollary}
	\begin{proof}
		For (a), our configuration dictates that the Koenigs domain of $(\phi_t)$ is necessarily an upper half-plane; cf. \cite[Proposition 9.3.12]{Booksem}. Lemma \ref{lm:dirichlet standard domain}(d) and Remark \ref{rem:dirichlet sectors} provide the desired spectrum. Moreover, (b) may be deduced exactly as Corollary \ref{cor:hyperbolic dirichlet}(b). Finally, suppose that $(\phi_t)$ is of finite shift. By \cite[Proposition 3.2]{our_finiteshift}, we know that the inner argument of $(\phi_t)$ is $\pi$. Clearly, this implies that the outer argument is $\pi$, as well. As a result, Theorem \ref{thm:intro phs dirichlet}(c) yields (c).
	\end{proof}
	
	Before we end the subsection, we lay out several examples illustrating the different cases that can occur in Theorem \ref{thm:intro phs dirichlet}. Lemma \ref{lm:dirichlet standard domain} already provides examples where the point spectrum is exactly an open angular sector plus the eigenvalue $0$. Below we mention some more involved examples which will help illustrate the utility of the ``blueprint'' provided by Theorem \ref{thm:intro phs dirichlet}.
	
	\begin{example}
		Set $\Omega=\{x+iy\in\C:x>0, \;0<y<\sqrt{x}\}$. This domain is simply connected, convex in the positive direction, is contained in the upper half-plane, and is not contained in any horizontal strip. Thus it corresponds to a parabolic semigroup $(\phi_t)$ of positive hyperbolic step. It can be seen that $\theta_\Omega=\Theta_\Omega=0$. So, by Theorem \ref{thm:intro phs dirichlet}(a), $\sigma_\mathcal{D}=iS(0,\pi)\cup\{0\}$.
	\end{example}
	
	In the case of hyperbolic semigroups, the induced point spectrum is either trivial, or a vertical strip bounded by the imaginary axis plus the eigenvalue $0$, or the left half-plane plus the eigenvalue $0$. So a natural conjecture for parabolic semigroups of positive hyperbolic step would be that the point spectrum is either trivial or an angular sector with the upper imaginary semi-axis as one of its sides, plus the eigenvalue $0$. Via the next example, we see that this is not the case.
	
	\begin{example}\label{ex:any strip}
		We will construct a semigroup where the induced point spectrum excluding the eigenvalue $0$ is non-empty and is neither a sector $iS(0,\Theta)$, nor a sector $iS(0,\Theta]$. Consider the domain 
		$$\Omega=\{x+iy\in\C:x>0,\; 0<y<e^{2\alpha x}\}, \quad \alpha>0.$$
		As usual, $\Omega$ corresponds to a triplet $(\phi_t)$, $(T_t)$, $\Gamma_\mathcal{D}$. It can be checked that $\theta_\Omega=\Theta_\Omega=\pi/2$. Hence, by Theorem \ref{thm:intro phs dirichlet}(b), we already know that $iS(0,\pi/2)\cup\{0\}\subseteq \sigma_\mathcal{D} \subseteq iS(0,\pi/2]\cup\{0\}$. So, we only need to verify exactly which $\lambda\in\C$ satisfying $\mathrm{Re}\lambda<0$ and $\mathrm{Im}\lambda=0$ belong to the point spectrum. Fix such a $\lambda$ and set $h$ the Koenigs function of $(\phi_t)$. Then
		\begin{equation*}
			\int\limits_{\D}\left|\left(e^{\lambda h}\right)'(z)\right|^2dA(z)=\int\limits_{\Omega}e^{2\mathrm{Re}(\lambda w)}dA(w)
			=\int\limits_{0}^{+\infty}\int\limits_{0}^{e^{2\alpha x}}e^{2\mathrm{Re}\lambda x}dydx
			=\int\limits_{0}^{+\infty}e^{2(\mathrm{Re}\lambda+\alpha)x}dx
		\end{equation*}
		and the latter integral is finite if and only if $\mathrm{Re}\lambda<-\alpha$. As a result, for all these $\lambda\in\C$, $e^{\lambda h}\in\mathcal{D}$ if and only if $\mathrm{Re}\lambda<-\alpha$. In conclusion,
		$$\sigma_\mathcal{D}=iS(0,\pi/2)\cup(-\infty,-\alpha)\cup\{0\}.$$
	\end{example}
	
	\begin{remark}
		An identical procedure, but this time considering $\Omega=\{x+iy\in\C:x>0,\; 0<y<e^{\sqrt{x}}\}$ provides a semigroup which induces a point spectrum equal to $iS(0,\pi/2]\cup\{0\}$.
	\end{remark}
	
	For our last example concerning parabolic semigroups of positive hyperbolic step, we will examine a semigroup where the inner and outer arguments are different.

	\begin{example}
		Fix $\alpha\in(0,\pi)$ and consider a sequence $\{a_n\}\subset\C$ with strictly increasing to $+\infty$ imaginary parts such that $\arg a_n=\alpha$, for all $n\in\mathbb{N}$. For each $n\in\mathbb{N}$ set $L_n=\{a_n-t:t\ge0\}$ and construct the domain $\Omega=S(0,\pi)\setminus\cup_{n\in\mathbb{N}}L_n$. Evidently $\Omega$ corresponds to a parabolic semigroup $(\phi_t)$ of positive hyperbolic step which induces $(T_t)$ with infinitesimal generator $\Gamma_\mathcal{D}$. By construction, $\theta_\Omega=\alpha$, while $\Theta_\Omega=\pi$. Since $\Omega$ contains a horizontal strip, standard arguments lead to $\sigma_\mathcal{D}=\{0\}$ and thus the outer argument ``prevails'' in this case.
	\end{example}
	
	We finish our study on the Dirichlet space with parabolic semigroups of zero hyperbolic step. Let $(\phi_t)$ be such a semigroup and denote by $h$ its Koenigs function and by $\Omega$ its Koenigs domain. Since $h$ is a univalent mapping of the unit disk, clearly $\Omega\subsetneq\C$ and hence $\partial\Omega$ has at least one point. In addition, $\Omega$ is convex in the positive direction, and hence $\partial\Omega$ contains at least a horizontal half-line stretching to infinity in the negative direction (i.e. with constant imaginary part and decreasing real part). Since $(\phi_t)$ is of zero hyperbolic step, $\Omega$ is not contained in any horizontal half-plane. Thus, we will always assume that $\Omega$ is contained in the slit plane $\C\setminus(-\infty,0]=S(-\pi,\pi)$. Since the geometry of a Koenigs domain is invariant with respect to translations, this normalization does not harm generality.
	
	Recall that for parabolic semigroups of zero hyperbolic step, we need all the notions of upper and lower, inner and outer arguments and that $\theta_\Omega=\theta_\Omega^-+\theta_\Omega^+$ and $\Theta_\Omega=\Theta_\Omega^-+\Theta_\Omega^+$. With the help of Proposition \ref{prop:dirichlet spectrum koenigs} and Lemma \ref{lm:dirichlet standard domain}, and using identical techniques as in the proof of Theorem \ref{thm:intro phs dirichlet}, one may prove Theorem \ref{thm:intro 0hs dirichlet}. For this reason, we opt not to include the proof. 
	
	As was the case with the other types of non-elliptic semigroups, Theorem \ref{thm:intro 0hs dirichlet} provides direct corollaries with respect to the dynamical properties of the initial semigroups. We refrain from explicitly providing them since they are in a very similar vein to those in Corollary \ref{cor:hyperbolic dirichlet} and Corollary \ref{cor:phs dirichlet}. Many examples akin to those in the previous two subsections may be constructed. Before ending the section, we will give one interesting example corresponding to Theorem \ref{thm:intro 0hs dirichlet}(c), as it yields a point spectrum that is not connected.
	
	\begin{example}\label{ex:not connected}
		Let $\Omega=\{x+iy\in\C:x>0,\;-e^x<y<e^x\}$. Obviously, $\Omega$ is a simply connected, convex in the positive direction domain and corresponds to a triplet $(\phi_t), (T_t), \Gamma_\mathcal{D}$. Clearly $\Omega$ is not contained in any horizontal half-plane, and hence $(\phi_t)$ is parabolic of zero hyperbolic step. Easy geometric considerations show that $\theta_\Omega=\Theta_\Omega=\pi$. By Theorem \ref{thm:intro 0hs dirichlet} we know that $\sigma_\mathcal{D}$ is contained in the closed negative semi-axis. Let $h$ denote the Koenigs function of $(\phi_t)$. Then, $\lambda\in\sigma_\mathcal{D}$ if and only if $e^{\lambda h}\in\mathcal{D}$. Certainly, from our discussion above, $\mathrm{Im}\lambda=0$. Moreover, since $0$ is necessarily an eigenvalue, we only care about the case $\mathrm{Re}\lambda<0$. Following the same steps as in preceding proofs and using Cartesian coordinates, $\lambda\in\sigma_\mathcal{D}\setminus\{0\}$ if and only if
		\begin{equation*}
			+\infty>\int\limits_{\Omega}e^{2\mathrm{Re}(\lambda w)}dA(w)=\int\limits_{0}^{+\infty}\int\limits_{-e^x}^{e^x}e^{2\mathrm{Re}\lambda x}dydx
			=2\int\limits_{0}^{+\infty}e^{(2\mathrm{Re}\lambda+1)x}dx.
		\end{equation*}
		But the latter integral is finite if and only if $2\mathrm{Re}\lambda+1<0$ which leads to $\sigma_\mathcal{D}=(-\infty,-1/2)\cup\{0\}$.
	\end{example}
	
	Of course, with a suitable slight modification (see e.g. Example \ref{ex:any strip}), we could have constructed the Koenigs domain in such a way that the point spectrum would be equal to $(-\infty,\alpha)\cup\{0\}$, for any $\alpha<0$, or even the whole interval $(-\infty,0]$. As a last outtake from this section, we understand that, in general, the point spectrum in the Dirichlet space is not connected. We will see in Section \ref{sec:bergman and hardy}, that this is not true in the setting of the Hardy and Bergman spaces, where the point spectrum is necessarily connected.

	\section{\quad Growth Estimates}\label{sec:growth bounds}
	In this brief section, we are going to leave point spectra aside and work with other notions in the study of operator theory. Given a semigroup $(T_t)$ of composition operators acting on a Banach space $X$, our first objective will be to estimate the norm $||T_t||_X$ of the operators $T_t:X\to X$, $t\ge0$. In place of $X$ we will put the Bergman spaces $A^p_\alpha$, $p\ge1$, $\alpha>-1$, and the Hardy spaces $H^p$, $p\ge1$. In both spaces, the semigroup $(T_t)$ is strongly continuous. So, it also makes sense to compute its type with respect to each space and the corresponding spectral radii. We disregard the Dirichlet space $\mathcal{D}$, since by \cite[Theorem 2]{Siskakis-Dirichlet}, given any semigroup $(\phi_t)$ in $\D$ which induces the semigroup $(T_t)$, the type $\omega_\mathcal{D}$ is always equal to $0$. As a consequence, the spectral radius $r_\mathcal{D}(T_t)$ of each composition operator is equal to $1$.
	
	\begin{remark}
		Let $(T_t)$ be a semigroup of composition operators acting on the Dirichlet space with infinitesimal generator $\Gamma_\mathcal{D}$. Denote by $\sigma_\mathcal{D}^f$ its full spectrum. Since the growth bound of $(T_t)$ is $0$, we know that $\sigma_\mathcal{D}^f$ lies in the closure of the left half-plane. In addition, $\sigma_\mathcal{D}^f$ is a closed set containing the corresponding point spectrum $\sigma_\mathcal{D}$. As a result, Theorem \ref{thm:intro hyperbolic dirichlet}(b), Theorem \ref{thm:intro phs dirichlet}(a) and Theorem \ref{thm:intro 0hs dirichlet}(a) provide a wide range of semigroups where the full spectrum is equal to $i\overline{S(0,\pi)}$. Note that the process of explicitly finding the full spectrum of an operator is at times extremely demanding.
	\end{remark}
	
	\hiddennumberedsubsection{Bergman Spaces}
	
	We start with the Bergman spaces $A^p_\alpha$, $p\ge1$, $\alpha>-1$. To simplify the notation, the norm with respect to the Bergman space $A^p_\alpha$ will be written as $||\cdot||_{p,\alpha}:=||\cdot||_{A^p_\alpha}$. Likewise, we set $\omega_{p,\alpha}:=\omega_{A^p_\alpha}$ and $r_{p,\alpha}:=r_{A^p_\alpha}$ the growth bound and the spectral radius, respectively.
	
	In order to prove our result, we will need the following fundamental inequality about the norm of a composition operator in a Bergman space.
	
	\begin{lemma}[{\cite[Section 11.3]{zhu}}]\label{lm:bergman estimate}
		Let $\phi:\D\to\D$ be holomorphic and consider the composition operator $T_\phi:A^p_\alpha\to A^p_\alpha$, $p\ge1$, $\alpha>-1$, with $T_\phi(f)=f\circ\phi$. Then
		$$\left(\dfrac{1}{1-|\phi(0)|^2}\right)^{(2+\alpha)/p} \le ||T_\phi||_{p,\alpha} \le \left(\dfrac{1+|\phi(0)|}{1-|\phi(0)|}\right)^{(2+\alpha)/p}.$$
	\end{lemma}
	
	Combining the inequalities of Lemma \ref{lm:bergman estimate} with the hyperbolic distance formula in \eqref{eq:hyperbolic distance 0}, we infer that given a holomorphic semigroup $(\phi_t)$ inducing $(T_t)$,
	\begin{equation}\label{eq:bergman hyperbolic estimate}
		\log\dfrac{1}{4}+\dfrac{2(2+\alpha)}{p}d_{\D}(0,\phi_t(0)) \le \log||T_t||_{p,\alpha} \le \dfrac{2(2+\alpha)}{p}d_{\D}(0,\phi_t(0)), \quad t\ge0.
	\end{equation}

	Theorem \ref{thm:total speed} and Lemma \ref{lm:total speed finite shift} allow us to extract the following result.

	\begin{corollary}\label{cor:type bergman}
		Let $(\phi_t)$ be a non-elliptic semigroup in $\D$ which induces the semigroup of composition operators $(T_t)$.
		\begin{enumerate}
			\item[\textup{(a)}] If $(\phi_t)$ is hyperbolic with spectral value $\mu>0$, then $\omega_{p,\alpha}=(2+\alpha)\mu/p$, for all $p\ge1$, $\alpha>-1$, and $r_{p,\alpha}(T_t)=e^{(2+\alpha)\mu t/p}$, for all $p\ge1$, $\alpha>-1$ and all $t\ge0$.
			\item[\textup{(b)}] If $(\phi_t)$ is parabolic, then $\omega_{p,\alpha}=0$, for all $p\ge1$, $\alpha>-1$, and $r_{p,\alpha}(T_t)=1$, for all $p\ge1$, $\alpha>-1$ and all $t\ge0$.
			\item[\textup{(c)}] If $(\phi_t)$ is of finite shift, then
			$$\lim\limits_{t\to+\infty}\dfrac{\log||T_t||_{p,\alpha}}{\log t}=\dfrac{2(2+\alpha)}{p}, \quad\text{for all }p\ge1,\; \alpha>-1.$$
		\end{enumerate}
	\end{corollary}

	\hiddennumberedsubsection{Hardy Spaces}
	
	We continue with the Hardy spaces $H^p$, $p\ge1$, of the unit disk. Let $(\phi_t)$ be a non-elliptic semigroup in $\D$ which induces the semigroup of composition operators $(T_t)$. Our first aim is to estimate the norm of each $T_t$, $t\ge0$, with respect to the Hardy spaces. In order to ease up the notation, we will just write $||T_t||_p:=||T_t||_{H^p}$, $t\ge0$, $p\ge1$.
	
	Analogously to the Bergman spaces, we are in need of a result estimating the norm of a composition operator in $H^p$.

	\begin{lemma}[{\cite[Corollary 3.7]{Cowen-Maccluer}}]\label{lm:hardy estimate}
		Let $\phi:\D\to\D$ be holomorphic and consider the composition operator $T_\phi:H^p\to H^p$, $p\ge1$, with $T_\phi(f)=f\circ\phi$. Then
		\begin{equation*}
			\left(\dfrac{1}{1-|\phi(0)|^2}\right)^{1/p} \le ||T_\phi||_p \le \left(\dfrac{1+|\phi(0)|}{1-|\phi(0)|}\right)^{1/p}.
		\end{equation*}
	\end{lemma}
	Applying Lemma \ref{lm:hardy estimate} on each $\phi_t$, $t\ge0$, and editing the inequality with the help of \eqref{eq:hyperbolic distance 0}, we see that
	\begin{equation}\label{eq:hardy hyperbolic estimate}
		\log\dfrac{1}{4}+\frac{2}{p}d_{\D}(0,\phi_t(0)) \le \log||T_t||_p \le \frac{2}{p}d_{\D}(0,\phi_t(0)), \quad t\ge0.
	\end{equation}
	
	Through this last inequality, we are able to calculate exactly the growth bound of any strongly continuous semigroup of composition operators acting on a Hardy space $H^p$. By extension, we can find explicitly the spectral radius of each composition operator in the semigroup. For the sake of convenience, we set $\omega_p:=\omega_{H^p}$ and $r_p:=r_{H^p}$. 
	If, in addition, the semigroup is induced by a holomorphic semigroup of finite shift, we proceed to an even better description. The proof of the result below may be directly derived via a combination of relation \eqref{eq:hardy hyperbolic estimate}, Theorem \ref{thm:total speed} and Lemma \ref{lm:total speed finite shift}. Parts (a) and (b) of the corollary below are also found in \cite[Theorem 3.9]{Cowen-Maccluer}. We include them for the sake of completeness.
	
	\begin{corollary}\label{cor:type hardy}
		Let $(\phi_t)$ be a non-elliptic semigroup in $\D$ which induces the semigroup of composition operators $(T_t)$.
		\begin{enumerate}
			\item[\textup{(a)}] If $(\phi_t)$ is hyperbolic with spectral value $\mu>0$, then $\omega_p=\mu/p$, for all $p\ge1$, and $r_p(T_t)=e^{\mu t/p}$, for all $p\ge1$ and all $t\ge0$.
			\item[\textup{(b)}] If $(\phi_t)$ is parabolic, then $\omega_p=0$, for all $p\ge1$, and $r_p(T_t)=1$, for all $p\ge1$ and all $t\ge0$.
			\item[\textup{(c)}] If $(\phi_t)$ is of finite shift, then
			$$\lim\limits_{t\to+\infty}\dfrac{\log||T_t||_p}{\log t}=\dfrac{2}{p}, \quad \text{for all }p\ge1.$$
		\end{enumerate}
	\end{corollary}

	\section{\quad Bergman and Hardy Spaces}\label{sec:bergman and hardy}
	
	We now proceed to the second main part of this work where we are going to examine the point spectrum of the infinitesimal generator whenever the semigroup of composition operators acts on a Bergman space $A_\alpha^p$, $p\ge1$, $\alpha>-1$, or on a Hardy space $H^p$, $p\ge1$. First of all, we state and prove two auxiliary results about the monotonicity properties and the shape of the point spectrum in both families of spaces. Then, we deal with the Bergman spaces using Green's function and the Hardy-Stein Formula from Theorem \ref{thm:bergman hardy-stein}. Finally, since Theorem \ref{thm:bergman hardy-stein} is a more general version of the Littlewood-Paley Formula in Theorem \ref{thm:hardy hardy-stein}, the results for the Hardy spaces will be derived as corollaries of our results on the Bergman spaces.
	
	Note that in \cite{Betsakos_Bergman} Betsakos fully characterized the point spectrum of the infinitesimal generator in the case of the Bergman space whenever the semigroup of composition operators is induced by a hyperbolic semigroup. Likewise, in \cite{betsakos_eig}, Betsakos completely describes the respective spectrum in the Hardy spaces whenever the initial semigroup is hyperbolic. Due to these works, we only work with the point spectra induced by parabolic semigroups. The problem of explicitly computing the point spectrum in such a case seems extremely hard. We will provide certain inclusion relations while also finding some sufficient conditions for equality. Following our procedure in the Dirichlet space, the main geometric tool will be the angular sectors containing or contained in the Koenigs domain. However, due to the increased complexity in the nature of the Bergman spaces, we are going to also use techniques of potential theory.
	
	\subsection{Basic Properties}
	Let $(T_t)$ be a semigroup of composition operators induced by a non-elliptic semigroup $(\phi_t)$. For the sake of brevity, if the $(T_t)$ acts on the Hardy space $H^p$, $p\ge1$, we will denote by $\Gamma_p$ its infinitesimal generator and by $\sigma_p:=\sigma_p(\Gamma_p)$ the corresponding point spectrum. If the semigroup acts on the Bergman space $A_\alpha^p$, $p\ge1$, $\alpha>-1$, we use the simplified notations $\Gamma_{p,\alpha}$ and $\sigma_{p,\alpha}:=\sigma_{p,\alpha}(\Gamma_{p,\alpha})$. 
	
	Let $h$ denote the Koenigs function of $(\phi_t)$. We know by \cite[p. 9]{siskakis_review} (see also \cite{siskakis_thesis}) that
	\begin{equation}\label{eq:spectrum hardy}
		\sigma_p=\{\lambda\in\C:e^{\lambda h}\in H^p\}, \quad p\ge1.
	\end{equation}
	On the other hand, by \cite[Theorem 2]{Siskakis-Bergman}, it is true that
	\begin{equation}\label{eq:spectrum bergman}
		\sigma_{p,\alpha}=\{\lambda\in\C:e^{\lambda h}\in A_\alpha^p\}, \quad p\ge1, \alpha>-1.
	\end{equation}
	Since the constant functions belong to all the Hardy and Bergman spaces, the number $0$ is always an eigenvalue and belongs trivially to all point spectra.
	
	We start with an auxiliary lemma correlating the size of the point spectrum $\sigma_p$ or $\sigma_{p,a}$ with the size of the Koenigs domain of $(\phi_t)$. This is the counterpart of Lemma \ref{lem:dirichlet spectrum monotonicity}.
	
	\begin{lemma}\label{lem:bergman spectrum monotonicity}
		Let $(\phi_t)$, $(\widetilde{\phi_t})$ be two non-elliptic semigroups in $\D$ with Koenigs functions $h$, $\widetilde{h}$, and Koenigs domains $\Omega$, $\widetilde{\Omega}$, respectively. Assume that $\Gamma_p$, $\Gamma_{p,\alpha}$ and $\widetilde{\Gamma_p}$, $\widetilde{\Gamma_{p,\alpha}}$ are the infinitesimal generators of the respective induced semigroups of composition operators in the Hardy space $H^p$ and the Bergman space $A_\alpha^p$. Denote by $\sigma_p$, $\sigma_{p,\alpha}$, $\widetilde{\sigma_p}$, and $\widetilde{\sigma_{p,\alpha}}$ the corresponding point spectra. If $\Omega\subseteq\widetilde{\Omega}$, then $\sigma_p\supseteq\widetilde{\sigma_p}$ and $\sigma_{p,\alpha}\supseteq\widetilde{\sigma_{p,\alpha}}$.
	\end{lemma}
	\begin{proof}
		We will only prove the second inclusion with regard to the Bergman space. The first one follows almost identically. Fix $p\ge1$ and $\alpha>-1$. Certainly $0\in\sigma_{p,\alpha}\cap\widetilde{\sigma_{p,\alpha}}$. Let $\lambda\in\widetilde{\sigma_{p,\alpha}}\setminus\{0\}$. Due to \eqref{eq:spectrum bergman}, we deduce that $e^{\lambda\widetilde{h}}\in A_\alpha^p$ and hence by Theorem \ref{thm:bergman hardy-stein}
		\begin{equation}\label{eq:bergman monotonicity 1}
			\int\limits_{\D}\left|e^{\lambda \widetilde{h}(z)}\right|^{p-2}\left|\left(e^{\lambda\widetilde{h}}\right)'(z)\right|^{2}\left(\log\frac{1}{|z|}\right)^{\alpha+2}dA(z)=\int\limits_{\D}\left|e^{p\lambda\widetilde{h}(z)}\right||\lambda \widetilde{h}'(z)|^2\left(\log\frac{1}{|z|}\right)^{\alpha+2}dA(z)<+\infty.
		\end{equation} 
		Notice that $\log(1/|z|)=g_{\D}(0,z)=g_{\widetilde{\Omega}}(\widetilde{h}(0),\widetilde{h}(z))$ due to \eqref{eq:green unit disk} and the conformal invariance of the Green's function. In addition, the mapping $\widetilde{h}$ is univalent and we may proceed to the change of variables $w=\widetilde{h}(z)$ in \eqref{eq:bergman monotonicity 1}. Consequently,
		\begin{equation}\label{eq:bergman monotonicity 2}
			\int\limits_{\widetilde{\Omega}}\left|e^{p\lambda w}\right|(g_{\widetilde{\Omega}}(\widetilde{h}(0),w))^{\alpha+2}dA(w)<+\infty.
		\end{equation}
		By our hypothesis, $\Omega\subseteq\widetilde{\Omega}$ and hence $h(0)\in\widetilde{\Omega}$. Consider a neighborhood $A$ of $\widetilde{\Omega}$ around $h(0)$. Keeping in mind that $g_{\widetilde{\Omega}}(\widetilde{h}(0),\widetilde{h}(0))=g_{\widetilde{\Omega}}(h(0),h(0))=+\infty$ and that the Green's function vanishes when one of the arguments tends to the boundary, the maximum principle for harmonic functions allows us to shrink the neighborhood $A$ sufficiently so that there exists an absolute constant $c>0$ with
		\begin{equation}\label{eq:bergman monotonicity 3}
			g_{\widetilde{\Omega}}(h(0),w) \le c g_{\widetilde{\Omega}}(\widetilde{h}(0),w), \quad \textup{for all }w\in\widetilde{\Omega}\setminus A.
		\end{equation}
		If needed, we shrink $A$ even more so that $A\subset\Omega$ and $h^{-1}(A)\subseteq D(0,1/3)$. The positivity of the integrand in \eqref{eq:bergman monotonicity 2}, relation \eqref{eq:bergman monotonicity 3} and the domain monotonicity property of the Green's function lead to
		\begin{equation}\label{eq:bergman monotonicity 4}
			+\infty>\int\limits_{\Omega\setminus A}\left|e^{p\lambda w}\right|(g_{\widetilde{\Omega}}(\widetilde{h}(0),w))^{\alpha+2}dA(w)\ge \frac{1}{c^{\alpha+2}} \int\limits_{\Omega\setminus A}\left|e^{p\lambda w}\right|(g_{\Omega}(h(0),w))^{\alpha+2}dA(w).
		\end{equation}
		Via the change of variables $z=h^{-1}(w)$, \eqref{eq:bergman monotonicity 4} combined with \eqref{eq:green unit disk} implies that
		\begin{equation}\label{eq:bergman monotonicity 5}
			\int\limits_{\D\setminus h^{-1}(A)}\left|e^{\lambda h(z)}\right|^{p-2}\left|\left(e^{\lambda h}\right)'(z)\right|^{2}\left(\log\frac{1}{|z|}\right)^{\alpha+2}dA(z)<+\infty.
		\end{equation}
		Finally, inside the simply connected domain $h^{-1}(A)$, both functions $e^{\lambda h}$ and $(e^{\lambda h})'$ are obviously bounded. In addition, the function $\log(1/|z|)^{\alpha+2}$ is integrable in the disk $D(0,1/3)$ for all $\alpha\ge-1$. Therefore, we also have that
		\begin{equation}\label{eq:bergman monotonicity 6}
			\int\limits_{h^{-1}(A)}\left|e^{\lambda h(z)}\right|^{p-2}\left|\left(e^{\lambda h}\right)'(z)\right|^{2}\left(\log\frac{1}{|z|}\right)^{\alpha+2}dA(z)<+\infty.
		\end{equation}
		A combination of \eqref{eq:bergman monotonicity 5} with \eqref{eq:bergman monotonicity 6} along with an application of Theorem \ref{thm:bergman hardy-stein} yields $e^{\lambda h}\in A_\alpha^p$ which in view of \eqref{eq:spectrum bergman} provides the inclusion of $\lambda$ in $\sigma_{p,\alpha}$. As a result, $\widetilde{\sigma_{p,\alpha}}\subseteq\sigma_{p,\alpha}$. An identical procedure with $\alpha=-1$ and the utilization of Theorem \ref{thm:hardy hardy-stein} provide the required outcome for the Hardy spaces.
	\end{proof}
	
	\begin{remark}
		The monotonicity property of the point spectrum with respect to the Koenigs domains for the Hardy spaces appears in \cite[Proposition 1]{betsakos_eig} by means of the subordination principle. However, in his result, Betsakos assumes that $h(0)=\widetilde{h}(0)$. So, other than generalizing the result for the Bergman spaces, we also lift this extra assumption.
	\end{remark}
	
	In our second basic result, we study the shape of the point spectrum in conjunction to its connectedness. We will prove that the point spectrum $\sigma_{p,\alpha}$ is \textit{starlike with respect to} $0$ for any valid choice of $p$, $\alpha$. The respective result for the Hardy spaces, whose proof we follow, may be found in \cite[Lemma 4.3]{gonzalezdona}.
	
	\begin{lemma}\label{lm:spectrum starlike}
		Let $(T_t)$ be a semigroup of composition operators induced by a non-elliptic semigroup in $\D$ and acting on the Bergman space $A_\alpha^p$, $p\ge1$, $\alpha>-1$. Then the point spectrum of its infinitesimal generator is starlike with respect to $0$.
	\end{lemma}
	\begin{proof}
		Let $\sigma_{p,\alpha}$ denote the point spectrum. As we already mentioned, it surely contains $0$. Let $\lambda\in\sigma_{p,\alpha}\setminus\{0\}$. Fix $t\in(0,1)$. We will estimate the Bergman norm of the function $e^{t\lambda h}$ where $h$ is the Koenigs function of the holomorphic semigroup inducing $(T_t)$. Set 
		$$\D^+:=\{z\in\D:\mathrm{Re}(\lambda h(z))>0\} \quad\textup{ and }\quad \D^-:=\{z\in\D:\mathrm{Re}(\lambda h(z))\le 0\}.$$
		Clearly $\D^+\cap\D^-=\emptyset$ whereas $\D^+\cup\D^-=\D$. Then
		\begin{eqnarray*}
			||e^{t\lambda h}||_{A^p_\alpha}^p&=&\int\limits_{\D}\left|e^{tp\lambda h(z)}\right|(1-|z|^2)^{\alpha}dA(z)\\
			&\le&\int\limits_{\D^+}\left|e^{p\lambda h(z)}\right|(1-|z|^2)^\alpha dA(z)+2^\alpha\int\limits_{\D^-}(1-|z|)^\alpha dA(z)\\
			&\le&||e^{\lambda h}||_{A_\alpha^p}^p+\frac{2^{\alpha+1}\pi}{(\alpha+1)(\alpha+2)}.
		\end{eqnarray*}
		Since $\lambda\in\sigma_{p, \alpha}$, \eqref{eq:spectrum bergman} yields $e^{\lambda h}\in A_{\alpha}^p$. The inequalities above then show that $e^{t\lambda h}\in A_\alpha^p$ as well. Therefore, \eqref{eq:spectrum bergman} implies that $t\lambda\in\sigma_{p, \alpha}$ and we have the desired result.
	\end{proof}
	
	As a direct corollary of Lemma \ref{lm:spectrum starlike}, we deduce that every spectrum $\sigma_{p,\alpha}$ is also a connected set. Recall that by Example \ref{ex:not connected}, this behavior is not maintained in the Dirichlet space.  
	
	Before we proceed to the main results, we state and prove one easy lemma about the shape of the point spectrum. However, this lemma applies only to infinitesimal generators induced by non-elliptic semigroups with positive hyperbolic step (either hyperbolic or parabolic of positive hyperbolic step). For the Hardy spaces, this result was proved by Betsakos in \cite{betsakos_eig}. Also, for hyperbolic semigroups, the following result is trivially true due to \cite{Betsakos_Bergman}. So, we only prove it for the Bergman spaces and parabolic semigroups of positive hyperbolic step. Recall that for such semigroups, we have introduced the normalization that their Koenigs domains are contained in the upper half-plane.
	\begin{lemma}\label{lm:convexity in the upper direction}
		Let $(T_t)$ be a semigroup of composition operators induced by a parabolic semigroup of positive hyperbolic step in $\D$ and acting on the Bergman space $A_\alpha^p$, $p\ge1$, $\alpha>-1$. If $\lambda$ belongs to the point spectrum of its infinitesimal generator, then so does $\lambda+it$, for all $t\ge0$.
	\end{lemma}
	\begin{proof}
		Fix $p\ge1$, $\alpha>-1$ and let $\sigma_{p,\alpha}$ denote the point spectrum of the infinitesimal generator. By \eqref{eq:spectrum bergman}, if $\lambda\in\sigma_{p,\alpha}$, then $e^{\lambda h}\in A_{\alpha}^p$, where $h$ is the Koenigs function of the semigroup inducing $(T_t)$. Thus, by definition, $||e^{\lambda h}||_{A_\alpha^p}<+\infty$. However, due to the shape of the respective Koenigs domain, $\mathrm{Im}h(z)>0$, for all $z\in\D$. As a result,
		\begin{equation*}
			\left|e^{(\lambda+it)h(z)}\right|=\left|e^{\lambda h(z)}\right| \, e^{-t\, \mathrm{Im}h(z)}\le \left|e^{\lambda h(z)}\right|,
		\end{equation*}
		for all $z\in\D$ and all $t\ge0$. Therefore, $\|e^{(\lambda+it)h}\|_{A^p_\alpha}\le \|e^{\lambda h}\|_{A^p_\alpha}<+\infty$. Ergo $e^{(\lambda+it)h}\in A_{\alpha}^p$, which leads to $\lambda+it\in\sigma_{p,\alpha}$ by \eqref{eq:spectrum bergman}.
	\end{proof}
	
	\begin{remark}
		If the Koenigs domain was contained in any other upper half-plane, the result holds following an identical procedure. On the other hand, if the Koenigs domain was contained in a lower half-plane, we would have that $\lambda-it\in\sigma_{p,\alpha}$, for all $t\ge0$. Besides, the number $0$ is always an eigenvalue. As a result, a direct corollary of the previous lemma is that the half-line $\{it:t\ge0\}$ (or $\{it:t\le0\}$ in the case of the lower half-plane) lies always completely in the point spectrum when the initial semigroup is parabolic of positive hyperbolic step.
	\end{remark}

	\subsection{Bergman Spaces}
	We move on to our main results concerning semigroups of composition operators in the Bergman spaces. Recall that for any semigroup of composition operators induced by a parabolic semigroup and any Bergman space, Corollary \ref{cor:type bergman} dictates that the type is equal to $0$. Therefore, the full spectrum of the infinitesimal generator of such a semigroup is contained in the closed left half-plane. By extension, 
	\begin{equation}\label{eq:spectrum bergman inclusion}
		\sigma_{p,\alpha}\subseteq i\overline{S(0,\pi)}, \quad \textup{for all }p\ge1 \textup{ and all }\alpha>-1.
	\end{equation} 
	
	We commence with the case where the Koenigs domain of the holomorphic semigroup is exactly an angular sector. We uphold the normalization of the previous sections and thus, consider all our Koenigs domains to be contained in the slit plane $\C\setminus(-\infty,0]$.
	
	\begin{lemma}\label{lm:bergman sector}
		Let $(\phi_t)$ be a semigroup in $\D$ with Koenigs domain $S(a,b)$, $-\pi\le a<b \le \pi$. Suppose that $(\phi_t)$ induces the semigroup of composition operators $(T_t)$ with infinitesimal generator $\Gamma_{p,\alpha}$.
		\begin{enumerate}
			\item[\textup{(a)}] If $b-a\in (0,\pi)$, then $\sigma_{p,\alpha}=i\overline{S(-a,\pi-b)}$.
			\item[\textup{(b)}] If $b-a=\pi$, then $\sigma_{p,\alpha}=\{ie^{-ia}t:t\ge0\}$.
			\item[\textup{(c)}] If $b-a \in(\pi,2\pi]$, then $\sigma_{p,\alpha}=\{0\}$.
		\end{enumerate}
	\end{lemma}
	\begin{proof}
		For $S(a,b)$ to be a Koenigs domain, it needs to be simply connected and convex in the positive direction. Therefore, it is necessary that $a\le 0\le b$ and $a<b$. Clearly, if one of the two is equal to $0$, then the angular sector is contained in a horizontal half-plane and $(\phi_t)$ is of positive hyperbolic step. In case $a<0<b$, $(\phi_t)$ is of zero parabolic step. To start the proof, we execute some calculations that adhere to all the possible values of $b-a$. The distinction will happen towards the end of the proof. Fix $p\ge1$ and $\alpha>-1$. In any case, $0\in\sigma_{p,\alpha}$. Now let $\lambda\in\C\setminus\{0\}$. In order to examine the inclusion of $\lambda$ in the spectrum, we need to examine the inclusion of $e^{\lambda h}$ in the Bergman space $A_\alpha^p$, where $h$ is the Koenigs function of the semigroup. By Theorem \ref{thm:bergman hardy-stein}, $e^{\lambda h}\in A_\alpha^p$ if and only if
		\begin{equation}\label{eq:bergman main 1}
			\int\limits_{\D}\left|e^{\lambda h(z)}\right|^{p-2}\left|\left(e^{\lambda h}\right)'(z)\right|^2\left(\log\frac{1}{|z|}\right)^{\alpha+2}dA(z)=\int\limits_{\D}\left|e^{\lambda h(z)}\right|^p|\lambda h'(z)|^2\left(\log\frac{1}{|z|}\right)^{\alpha+2}dA(z)<+\infty.
		\end{equation}
		Using the change of variables $w=h(z)$ and \eqref{eq:green unit disk} along with the conformal invariance of the Green's function, we understand that $\lambda\in\sigma_{p,\alpha}$ if and only if 
		\begin{equation}\label{eq:bergman main 2}
			\int\limits_{S(a,b)}\left|e^{p\lambda w}\right|(g_{\D}(0,h^{-1}(w)))^{\alpha+2}dA(w)=\int\limits_{S(a,b)}\left|e^{p\lambda w}\right|(g_{S(a,b)}(h(0),w))^{\alpha+2}dA(w)<+\infty.
		\end{equation}
		As in the previous lemmas, the singularity for $w=h(0)$ due to the Green's function is integrable and we only need to estimate the integral ``close'' to the boundary of the sector. By conformal invariance, it is easy to see that 
		\begin{equation*}
			g_{S(a,b)}(h(0),w)=g_{S((a-b)/2,(b-a)/2)}(e^{-i(a+b)/2}h(0),e^{-i(a+b)/2}w), 
		\end{equation*}
		for any $w\in S(a,b)$. Through elementary computations, the Koenigs function $h$ can be chosen so that $e^{-i(a+b)/2}h(0)=1$. Then, in view of the symmetry of the Green's function and of Example \ref{ex:green sector}, setting $w=re^{i\theta}$, we see that
		\begin{eqnarray}\label{eq:bergman main 3}
			\notag	g_{S((a-b)/2,(b-a)/2)}(e^{-i(a+b)/2}h(0),e^{-i(a+b)/2}w)&=&\log\left|\dfrac{(e^{-i(a+b)/2}w)^{\pi/(b-a)}+1}{(e^{-i(a+b)/2}w)^{\pi/(b-a)}-1}\right| \\
			&=&\frac{1}{2}\log\dfrac{r^{2\pi/(b-a)}+2r^{\pi/(b-a)}\cos\psi_\theta+1}{r^{2\pi/(b-a)}-2r^{\pi/(b-a)}\cos\psi_\theta+1},
		\end{eqnarray}
		where $\psi_\theta:=\frac{\pi}{b-a}(\theta-\frac{a+b}{2})$. It is straightforward that $\cos\psi_\theta>0$ since $\theta\in(a,b)$. Applying the inequality $1-\frac{1}{x}\le \log x \le x-1$ which is valid for all $x>0$, relation \eqref{eq:bergman main 3} yields
		\begin{equation*}
			\dfrac{2r^{\frac{\pi}{b-a}}\cos\psi_\theta}{\left(r^{\frac{\pi}{b-a}}+1\right)^2} \le g_{S(a,b)}(h(0),w) \le \dfrac{2r^{\frac{\pi}{b-a}}\cos\psi_\theta}{\left(r^{\frac{\pi}{b-a}}-1\right)^2}.
		\end{equation*}
		A little reformulation on the latter inequality implies that
		\begin{equation}\label{eq:bergman main 4}
			\frac{1}{2}r^{-\frac{\pi}{b-a}}\cos\psi_\theta \le g_{S(a,b)}(h(0),w) \le 8 r^{-\frac{\pi}{b-a}}\cos\psi_\theta,
		\end{equation}
		for all $\theta\in(a,b)$ and all $r\in[2^{(b-a)/\pi},+\infty)$. Set $r_0=2^{(b-a)/\pi}>1$. Returning to \eqref{eq:bergman main 2} and using polar coordinates, we get that $\lambda\in\sigma_{p,\alpha}$ if and only if 
		\begin{equation}\label{eq:bergman main 5}
			\int\limits_{a}^{b}\int\limits_{0}^{+\infty}re^{pr|\lambda|\cos(\arg\lambda+\theta)}(g_{S(a,b)}(h(0),re^{i\theta}))^{\alpha+2}drd\theta<+\infty.
		\end{equation}
		Clearly, the finiteness of the integral is independent of $|\lambda|$ and only depends on $\arg\lambda$. In addition, the part of the integral for $r\in(0,r_0)$ is finite. So the finiteness hinges on the behavior for large $r$. However, from \eqref{eq:bergman main 4}, the function $g_{S(a,b)}(h(0),re^{i\theta})$ is comparable to $r^{-\pi/(b-a)}\cos\psi_\theta$, for $r\in[r_0,+\infty)$. Thus, returning to \eqref{eq:bergman main 5}, we deduce that $\lambda\in\sigma_{p,\alpha}$ if and only if
		\begin{equation}\label{eq:bergman main 6}
			\int\limits_{a}^{b}\int\limits_{r_0}^{+\infty}re^{pr|\lambda|\cos(\arg\lambda+\theta)}\dfrac{(\cos\psi_\theta)^{\alpha+2}}{r^{\frac{\pi(\alpha+2)}{b-a}}}drd\theta=\int\limits_{a}^{b}(\cos\psi_\theta)^{\alpha+2}\int\limits_{r_0}^{+\infty}r^{1-\frac{\pi(\alpha+2)}{b-a}}e^{pr|\lambda|\cos(\arg\lambda+\theta)}drd\theta<+\infty.
		\end{equation}
		Similarly to the Dirichlet space, the latter integral is finite if and only if $\cos(\arg\lambda+\theta)<0$ for all $\theta\in(a,b)$. This leads to $\arg\lambda\in[\pi/2-a,3\pi/2-b]$. Having this as a stepping stone, we start distinguishing cases:
		
		(a) Suppose that $b-a\in(0,\pi)$. Assume that $\arg \lambda \in[\pi/2-a,3\pi/2-b]$ (note that this interval has non-empty interior). The exponent of $r$ in the latter integral of \eqref{eq:bergman main 6} is negative. As a result, and since $r_0>1$,
		\begin{eqnarray*}
			\int\limits_{a}^{b}(\cos\psi_\theta)^{\alpha+2}\int\limits_{r_0}^{+\infty}r^{1-\frac{\pi(\alpha+2)}{b-a}}e^{pr|\lambda|\cos(\arg\lambda+\theta)}drd\theta &\le& \int\limits_{a}^{b}(\cos\psi_\theta)^{\alpha+2}\int\limits_{r_0}^{+\infty}e^{pr|\lambda|\cos(\arg\lambda+\theta)}drd\theta\\
			&=&-\frac{1}{r_0|\lambda|}\int\limits_{a}^{b}(\cos\phi_\theta)^{\alpha+2}\frac{e^{pr_0|\lambda|\cos(\arg\lambda+\theta)}}{\cos(\arg\lambda+\theta)}d\theta.
		\end{eqnarray*}
		Taking limits as $\theta\to a$ and as $\theta\to b$, it can be verified that the integrand is bounded for all choices of $\arg\lambda\in[\pi/2-a,3\pi/2-b]$ and for any modulus $|\lambda|$. Therefore, by \eqref{eq:spectrum bergman inclusion}, we get that $\sigma_{p,\alpha}=\overline{\{\lambda\in\C:\pi/2-a <\arg\lambda <3\pi/2-b\}}$.
		
		(b) Suppose that $b-a=\pi$. Working exactly as in the previous result we reach the same conclusion. However, this time $\pi/2-a=3\pi/2-b$ and the point spectrum transforms into a half-line.
		
		(c) Suppose that $b-a\in(\pi,2\pi]$. Then, for any choice of $\lambda\in\C\setminus\{0\}$ we may find a sufficiently small interval $(\theta_1,\theta_2)\subset(a,b)$ such that $\cos(\arg\lambda+\theta)>0$, for all $\theta\in(\theta_1,\theta_2)$. Hence the integral in \eqref{eq:bergman main 6} cannot be finite. Ergo $\sigma_{p,\alpha}=\{0\}$.
	\end{proof}
	
	\begin{remark}
		Evidently, if the Koenigs domain is of the form $q+S(a,b)$, for some $q\ne0$, the same result holds. This is because the vertex $q$ will just create an additive constant in the Koenigs function of the semigroup and by extension a multiplicative constant for $e^{\lambda h}$. But clearly such a multiplicative constant does not affect the finiteness of the Bergman norms.
	\end{remark}
	
	Using the latter lemma we can prove a general result about the point spectrum of infinitesimal generators in Bergman spaces. Note that our result does not solve the problem of characterizing the point spectrum, other than in certain cases. However, the inclusions below combined with the introductory lemmas of the previous subsection paint a clearer picture towards this objective. 
	
	We do not treat parabolic semigroups in a unified way, because Lemma \ref{lm:convexity in the upper direction} allows a slightly better description in the case of positive hyperbolic step. 
	For such semigroups the Koenigs domain is always contained in the classic upper half-plane and hence only the general notions of inner and outer arguments are needed. Once again, if the Koenigs domain is contained in any other upper or lower half-plane, the same inclusions hold albeit with the obvious modifications.

	\begin{theorem}\label{thm:bergman phs}
		Let $(\phi_t)$ be a parabolic semigroup of positive hyperbolic step in $\D$ which induces the semigroup of composition operators $(T_t)$. Suppose that $\Omega$ is the Koenigs domain of $(\phi_t)$ and that $\sigma_{p,\alpha}$ is the point spectrum of the infinitesimal generator of $(T_t)$ when acting on the Bergman space $A_\alpha^p$, $p\ge1$, $\alpha>-1$. Then:
		\begin{enumerate}
			\item[\textup{(a)}] If $\theta_\Omega=\Theta_\Omega=0$, then $iS[0,\pi)\cup\{0\}\subseteq\sigma_{p,\alpha}\subseteq i\overline{S(0,\pi)}$.
			\item[\textup{(b)}] If $\theta_\Omega=\Theta_\Omega\in(0,\pi)$, then $iS[0,\pi-\theta_\Omega)\cup\{0\}\subseteq \sigma_{p,\alpha} \subseteq i\overline{S(0,\pi-\theta_\Omega)}$.
			\item[\textup{(c)}] If $\theta_\Omega=\Theta_\Omega=\pi$, then $\sigma_{p,\alpha}=\{it:t\ge0\}$.
			\item[\textup{(d)}] If $\theta_\Omega<\Theta_\Omega$, then $iS[0,\pi-\Theta_\Omega)\cup\{0\}\subseteq\sigma_{p,\alpha}\subseteq i\overline{S(0,\pi-\theta_\Omega)}$. In case $\Theta_\Omega=\pi$, the set on the left becomes the vertical half-line $\{it:t\ge0\}$, while if $\theta_\Omega=0$, the set on the right becomes $i\overline{S(0,\pi)}$.
		\end{enumerate}
	\end{theorem}
	\begin{proof}
		Fix $p\ge1$ and $\alpha>-1$. We will treat each case with the help of Lemmas \ref{lm:convexity in the upper direction} and \ref{lm:bergman sector} depending on the shape of $\Omega$. Recall that $\omega_{p,\alpha}=0$ and hence $\sigma_{p,\alpha}\subseteq i\overline{S(0,\pi)}$ in general.
		
		(a) Since $\Theta_\Omega=0$, there exists a decreasing sequence $\{a_n\}\subset(0,\pi]$ satisfying $\lim_{n\to+\infty}a_n=0$ such that for each $n\in\mathbb{N}$ there exists $p_n\in\C$ with $\Omega\subset p_n+S(0,a_n)$, for all $n\in\mathbb{N}$. Combining Lemma \ref{lem:bergman spectrum monotonicity} and Lemma \ref{lm:bergman sector}, we obtain $\sigma_{p,\alpha}\supset i\overline{S(0,\pi-a_n)}$ for all $n\in\mathbb{N}$. Since the latter sequence of sets is increasing, taking its union and keeping in mind the asymptotic behavior of $\{a_n\}$, we obtain $\sigma_{p,\alpha}\supseteq iS(0,\pi)$. But $0\in \sigma_{p,\alpha}$ which by Lemma \ref{lm:convexity in the upper direction} yields $\{it:t\ge0\}\subset\sigma_{p,\alpha}$. Consequently, $\sigma_{p,\alpha}\subseteq i S[0,\pi)$. Due to the general inclusion $\sigma_{p,\alpha}\subseteq i\overline{S(0,\pi)}$, the desired outcome follows.
		
		(b) By our hypothesis, there exist an increasing sequence $\{a_n\}\subset(0,\pi]$, a decreasing sequence $\{b_n\}\subset(0,\pi]$ with $\lim_{n\to+\infty}a_n=\lim_{n\to+\infty}b_n=\theta_\Omega$ and two sequences of points $\{p_n\}, \{q_n\}\subset\C$ such that $p_n+S(0,a_n) \subset \Omega \subset q_n+S(0,b_n)$, for all $n\in\mathbb{N}$. By Lemma \ref{lm:bergman sector} and its subsequent remark, we understand that $i\overline{S(0,\pi-b_n)}\subset\sigma_{p,\alpha}\subset i\overline{S(0,\pi-a_n)}$, for all $n\in\mathbb{N}$. Thinking exactly as in the previous case, the convexity of the point spectrum in the upper direction implies that $\sigma_{p,\alpha}\subseteq iS[0,\pi-\theta_\Omega)$. On the other hand, the sets $\overline{S(0,\pi-a_n)}$ form a nested sequence and thus taking the intersection, $\sigma_{p,\alpha}\subseteq i\overline{S(0,\pi-\theta_\Omega)}$.
		
		(c) Since $\theta_\Omega=\pi$, there exists an increasing sequence $\{a_n\}\subset(0,\pi]$ converging to $\pi$ and a sequence of points $\{p_n\}\subset\C$ such that $\Omega\supset p_n+iS(0,a_n)$, for all $n\in\mathbb{N}$. By Lemma \ref{lem:bergman spectrum monotonicity}, we get $\sigma_{p,\alpha}\subset i\overline{S(0,\pi-a_n)}$, for all $n\in\mathbb{N}$. Considering the intersection, we deduce that $\{it:t\ge0\}=\sigma_{p,\alpha}$.
		
		(d) A combination of techniques used in all the previous cases leads to the desired inclusion.
	\end{proof}
	
	Theorem \ref{thm:bergman phs} does not provide a characterization of the point spectrum, but provides a ``map'' according to which one knows where to ``look'' in order to compute the spectrum. For example, if the Koenigs domain is $\Omega=\{x+iy\in\C:x>1,\; 0<y<\log x\}$, then $\theta_\Omega=\Theta_\Omega=0$. Therefore, in order to fully compute the point spectrum, in view of Theorem \ref{thm:bergman phs}(a), we just need to examine whether the point $it$, $t<0$, belong to $\sigma_{p,\alpha}$.
	
	Next, we turn our attention to parabolic semigroups of zero hyperbolic step. Unfortunately, in this case we do not have a convexity result at our disposal for the point spectrum. Hence, the inclusions below are not as strict. Since we are now not contained in any horizontal half-plane, we are in need of upper and lower, inner and outer arguments. The proof follows the same logic as in the previous theorem. For the sake of avoiding repetition, we omit it.

	\begin{theorem}\label{thm:bergman 0hs}
		Let $(\phi_t)$ be a parabolic semigroup of zero hyperbolic step in $\D$ which induces the semigroup of composition operators $(T_t)$. Suppose that $\Omega$ is the Koenigs domain of $(\phi_t)$ and that $\sigma_{p,\alpha}$ is the point spectrum of the infinitesimal generator of $(T_t)$ when acting on the Bergman space $A_\alpha^p$, $p\ge1$, $\alpha>-1$. Then:
		\begin{enumerate}
			\item[\textup{(a)}] If $\theta_\Omega=\Theta_\Omega=0$, then $iS(0,\pi)\cup\{0\}\subseteq\sigma_{p,\alpha}\subseteq i\overline{S(0,\pi)}$.
			\item[\textup{(b)}] If $\theta_\Omega=\Theta_\Omega\in(0,\pi)$ (and hence $\theta_\Omega^-=\Theta_\Omega^-, \theta_\Omega^+=\Theta_\Omega^+)$, then
			$$iS(\theta_\Omega^-,\pi-\theta_\Omega^+)\cup\{0\}\subseteq \sigma_{p,\alpha} \subseteq i\overline{S(\theta_\Omega^-,\pi-\theta_\Omega^+)}.$$
			\item[\textup{(c)}] If $\theta_\Omega=\pi$, then $\sigma_{p,\alpha}\subseteq \{ie^{i\theta_\Omega^-}t:t\ge0\}$.
			\item[\textup{(d)}] If $\theta_\Omega\in(\pi,2\pi]$, then $\sigma_{p,\alpha}=\{0\}$.
			\item[\textup{(e)}] If $\theta_\Omega<\min\{\pi,\Theta_\Omega\}$, then $iS(\Theta_\Omega^-,\pi-\Theta_\Omega^+)\cup\{0\}\subseteq\sigma_{p,\alpha}\subseteq i\overline{S(\theta_\Omega^-,\pi-\theta_\Omega^+)}$. In case $\Theta_\Omega\ge\pi$, the set on the left becomes $\{0\}$, while if $\theta_\Omega=0$, the set on the right becomes $i\overline{S(0,\pi)}$.
		\end{enumerate} 
	\end{theorem}
	
	Once again, Theorem \ref{thm:bergman 0hs} provides a helpful guide. For instance, if we have the Koenigs domain $\Omega=\{x+iy:x>1,\; y<|\log x|\}$, it is easy to see that $\theta_\Omega=\Theta_\Omega=0$. Consequently, in view of Theorem \ref{thm:bergman 0hs}(a), we just need to examine whether the points $it$, $t\in\mathbb{R}\setminus\{0\}$, belong to $\sigma_{p,\alpha}$ in order to have a complete characterization.
	
	\begin{remark}\label{rem:equality when outer argument is minimum}
		Let $(\phi_t)$ be a parabolic semigroup with Koenigs domain $\Omega$ such that $\theta_\Omega=\Theta_\Omega<\pi$. Recall that by definition, the outer argument of $(\phi_t)$ is in general an infimum. Evidently, there is a wide variety of convex in the positive direction, simply connected domains where this infimum is in fact a minimum. Suppose that for the given $(\phi_t)$ and $\Omega$, $\Theta_\Omega$ is indeed a minimum. As a result, there exists $q\in\C$ such that $\Omega\subseteq q+S(-\Theta_\Omega^-,\Theta_\Omega^+)$. Then, combining Proposition \ref{lem:bergman spectrum monotonicity} and Lemma \ref{lm:bergman sector}, we deduce $\sigma_{p,\alpha}\supseteq i\overline{S(\Theta_\Omega^-,\pi-\Theta_\Omega^+)}$, for all $p\ge1$ and all $\alpha>-1$. However, by Theorems \ref{thm:bergman phs} and \ref{thm:bergman 0hs}, the reverse inclusion also holds and we obtain equality for the point spectrum. Similarly, if $\theta_\Omega=\Theta_\Omega=\pi$ and the outer argument is a minimum, we have equality once again, and the point spectrum is a half-line. For instance, for the semigroup with Koenigs domain $\Omega=\{x+iy\in\C:x>0,\; 0<y<e^x\}$, we obtain at once $\sigma_{p,\alpha}=i\overline{S(0,\pi/2)}$.
	\end{remark}
	
	We conclude our study on Bergman spaces with a final result about semigroups of finite shift.
	
	\begin{corollary}\label{cor:bergman finite shift}
		Let $(\phi_t)$ be a semigroup of finite shift in $\D$ which induces the semigroup of composition operators $(T_t)$ acting on the Bergman space $A_\alpha^p$, $p\ge1$, $\alpha>-1$. Then
		$\sigma_{p,\alpha}=\{it:t\ge0\}$.
	\end{corollary}
	\begin{proof}
		As we mentioned in Subsection \ref{sub:semigroups}, a semigroup of finite shift is necessarily parabolic of positive hyperbolic step and has inner (and by extension outer) argument equal to $\pi$. Hence, $(\phi_t)$ falls into the category described in Theorem \ref{thm:bergman phs}(c) and the result follows.
	\end{proof}

	\subsection{Hardy Spaces}
	We conclude our study with the Hardy spaces $H^p$, $p\ge1$. A careful consideration of the proof of Lemma \ref{lm:bergman sector} shows that the proof still holds if in the application of the Hardy-Stein Formula we allow $\alpha=-1$. In this way, we obtain the Littlewood-Paley Formula in Theorem \ref{thm:hardy hardy-stein} and hence Lemma \ref{lm:bergman sector} holds identically for Hardy spaces.
	
	\begin{lemma}\label{lm:hardy sector}
		Let $(\phi_t)$ be a semigroup in $\D$ with Koenigs domain $S(a,b)$, $-\pi\le a<b\le\pi$. Suppose that $(\phi_t)$ induces the semigroup of composition operators $(T_t)$ acting on the Hardy space $H^p$, $p\ge1$, and the point spectrum of its infinitesimal generator is $\sigma_p$.
		\begin{enumerate}
			\item[\textup{(a)}] If $b-a\in(0,\pi)$, then $\sigma_p=i\overline{S(-a,\pi-b)}$.
			\item[\textup{(b)}] If $b-a=\pi$, then $\sigma_p=\{ie^{-ia}t:t\ge0\}$.
			\item[\textup{(c)}] If $b-a\in(\pi,2\pi]$, then $\sigma_p=\{0\}$.
		\end{enumerate}
	\end{lemma}
	
	In the case where $\Omega$ is an angular sector symmetric with respect to the real axis of inner argument less or equal to $\pi$, the full spectrum of the infinitesimal generator when $(T_t)$ acts on $H^2$ is completely calculated in \cite[Corollary 7.42]{Cowen-Maccluer}. In combination with our results, it follows that in this case the full spectrum and the point spectrum of the infinitesimal generator coincide.
	
	Using the above lemma as a starting point, we may extract the following counterparts of Theorems \ref{thm:bergman phs} and \ref{thm:bergman 0hs} for the Hardy spaces. The proofs are exactly the same and rely on Lemmas \ref{lem:bergman spectrum monotonicity} and \ref{lm:hardy sector} along with the convexity result found in \cite[Proposition 5]{betsakos_eig}. Recall that for any semigroup $(T_t)$ acting on $H^p$, $p\ge1$, and induced by a parabolic semigroup, its type is always equal to $0$ as seen in Corollary \ref{cor:type hardy}. Therefore, the full spectrum and the point spectrum of $\Gamma_p$ are trivially contained in the closed left half-plane.
	
	\begin{theorem}\label{thm:hardy phs}
		Let $(\phi_t)$ be a parabolic semigroup of positive hyperbolic step in $\D$ which induces the semigroup of composition operators $(T_t)$. Suppose that $\Omega$ is the Koenigs domain of $(\phi_t)$ and that $\sigma_p$ is the point spectrum of the infinitesimal generator of $(T_t)$ when acting on the Hardy space $H^p$, $p\ge1$. Then:
		\begin{enumerate}
			\item[\textup{(a)}] If $\theta_\Omega=\Theta_\Omega=0$, then $iS[0,\pi)\cup\{0\}\subseteq\sigma_p\subseteq i\overline{S(0,\pi)}$.
			\item[\textup{(b)}] If $\theta_\Omega=\Theta_\Omega\in(0,\pi)$, then $iS[0,\pi-\theta_\Omega)\cup\{0\}\subseteq \sigma_p \subseteq i\overline{S(0,\pi-\theta_\Omega)}$.
			\item[\textup{(c)}] If $\theta_\Omega=\Theta_\Omega=\pi$, then $\sigma_p=\{it:t\ge0\}$.
			\item[\textup{(d)}] If $\theta_\Omega<\Theta_\Omega$, then $iS[0,\pi-\Theta_\Omega)\cup\{0\}\subseteq\sigma_p\subseteq i\overline{S(0,\pi-\theta_\Omega)}$. In case $\Theta_\Omega=\pi$, the set on the left becomes the vertical half-line $\{it:t\ge0\}$, while if $\theta_\Omega=0$, the set on the right becomes $i\overline{S(0,\pi)}$.
		\end{enumerate}
	\end{theorem}
	
	\begin{theorem}\label{thm:hardy 0hs}
		Let $(\phi_t)$ be a parabolic semigroup of zero hyperbolic step in $\D$ which induces the semigroup of composition operators $(T_t)$. Suppose that $\Omega$ is the Koenigs domain of $(\phi_t)$ and that $\sigma_p$ is the point spectrum of the infinitesimal generator of $(T_t)$ when acting on the Hardy space $H^p$, $p\ge1$. Then:
		\begin{enumerate}
			\item[\textup{(a)}] If $\theta_\Omega=\Theta_\Omega=0$, then $iS(0,\pi)\cup\{0\}\subseteq\sigma_p\subseteq i\overline{S(0,\pi)}$.
			\item[\textup{(b)}] If $\theta_\Omega=\Theta_\Omega\in(0,\pi)$ (and hence $\theta_\Omega^-=\Theta_\Omega^-, \theta_\Omega^+=\Theta_\Omega^+)$, then
			$$iS(\theta_\Omega^-,\pi-\theta_\Omega^+)\cup\{0\}\subseteq \sigma_p \subseteq i\overline{S(\theta_\Omega^-,\pi-\theta_\Omega^+)}.$$
			\item[\textup{(c)}] If $\theta_\Omega=\pi$, then $\sigma_p\subseteq \{ie^{i\theta_\Omega^-}t:t\ge0\}$.
			\item[\textup{(d)}] If $\theta_\Omega\in(\pi,2\pi]$, then $\sigma_p=\{0\}$.
			\item[\textup{(e)}] If $\theta_\Omega<\min\{\pi,\Theta_\Omega\}$, then $iS(\Theta_\Omega^-,\pi-\Theta_\Omega^+)\cup\{0\}\subseteq\sigma_p\subseteq i\overline{S(\theta_\Omega^-,\pi-\theta_\Omega^+)}$. In case $\Theta_\Omega\ge\pi$, the set on the left becomes $\{0\}$, while if $\theta_\Omega=0$, the set on the right becomes $i\overline{S(0,\pi)}$.
		\end{enumerate} 
	\end{theorem}
	
	Summing up, Theorems \ref{thm:bergman phs} and \ref{thm:hardy phs} constitute Theorem \ref{thm:intro bergman phs}, while Theorems \ref{thm:bergman 0hs} and \ref{thm:hardy 0hs} constitute Theorem \ref{thm:intro bergman 0hs}.
	
	\begin{remark}The counterpart of Remark \ref{rem:equality when outer argument is minimum} holds for the Hardy spaces as well. Note that these two remarks do not apply to the Dirichlet space since the point spectrum of the infinitesimal generator is not closed when the Koenigs domain is an angular sector.
	\end{remark}
	
	Similar results about the Hardy spaces recently appeared in \cite{carlos-javi}. In this article, the authors use harmonic measure to acquire results on the point spectrum $\sigma_p$, $p\ge1$, whenever the initial semigroup is parabolic. Their study is also done with respect to angular sectors containing and contained inside the Koenigs domain, although they restrict to the case where the outer argument is strictly positive. Among this subclass of Koenigs domains, they further proceed to an in-depth study in the special case of convex domains.
	
	One direct corollary of Theorem \ref{thm:hardy phs} is once more the complete characterization of the point spectrum whenever the initial holomorphic semigroup is of finite shift.
	
	\begin{corollary}\label{cor:hardy finite shift}
		Let $(\phi_t)$ be a semigroup of finite shift in $\D$ which induces the semigroup of composition operators $(T_t)$ acting on the Hardy space $H^p$, $p\ge1$. Then, for the point spectrum of its infinitesimal generator, we have $\sigma_p=\{it:t\ge0\}$
	\end{corollary}
	
	We end the article with one non-trivial example where we explicitly calculate the point spectrum. Through this example, we demonstrate the potential complexity in the shape of the point spectrum with respect to the Hardy spaces induced by a parabolic semigroup, even in the case of positive hyperbolic step.
	
	\begin{figure}[ht]
		\centering
		\includegraphics[width=1\linewidth]{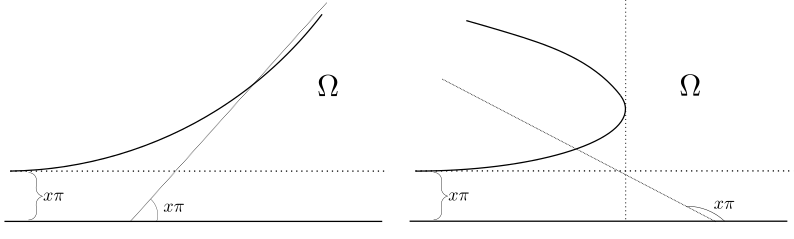}
		\caption{The domain $\Omega$ defined by $\psi_{\Omega}$ in Example \ref{ex:for figure}, for $x\in(0,1/2]$ on the left and for $x\in(1/2,1)$ on the right.} 
		\label{fig:Omega_x}
	\end{figure}
	\begin{example}\label{ex:for figure}
		Let $x\in(0,1)$ and consider the Koenigs domain $\Omega$ produced by the defining function $\psi_\Omega:(0,+\infty)\to[-\infty,+\infty)$ with
		$$\psi_\Omega(y)=\begin{cases}
			-\infty, \quad y\in(0,x\pi]\\
			(y-x\pi)\cot(x\pi)+\log(y-x\pi)-\log\sin(x\pi), \quad y\in(x\pi,+\infty).
		\end{cases}$$
		Clearly $\psi_\Omega$ is upper semi-continuous and $\Omega$ is well-defined as a Koenigs domain corresponding to a parabolic semigroup $(\phi_t)$ of positive hyperbolic step, due to the domain of $\psi_\Omega$. One may compute that the corresponding Koenigs function $h:\D\to\Omega$ is
		$$h(z)=\left(i\frac{1+z}{1-z}\right)^x+\log\left(i\frac{1+z}{1-z}\right)^x, \quad z\in\D,$$
		and that $\Theta_\Omega=\pi$, while $\theta_\Omega=x\pi$, as seen in Figure \ref{fig:Omega_x}. Also, for any choice of $x\in(0,1)$, $\Omega$ always contains the maximal horizontal strip $\{w\in\C:0<\mathrm{Im}w<x\pi\}$. Let $\sigma_p$ be the point spectrum of the induced semigroup of composition operators acting on the Hardy space $H^p$, $p\ge1$. At once, combining Theorem \ref{thm:hardy phs}(b) and \cite[Theorem 1(b)]{betsakos_eig}, we have that 
		\begin{equation}\label{eq:example1}
			\sigma_p\subseteq\left\{\lambda \in \mathbb{C}: -\frac{1}{px}<\mathrm{Re}\lambda\leq 0\right\}\cap i\overline{S(0,\pi-x\pi)}.
		\end{equation}
		We now examine the converse inclusion. So we fix $\lambda$ in the right-hand side of \eqref{eq:example1} and we study the integral of 
		\begin{align*}
			|e^{\lambda p h(z)}|&=\left|\exp\left\{p|\lambda|\left|\frac{1+z}{1-z}\right|^{x}e^{i(\frac{\pi x}{2}+\arg(\lambda)+x\arg(\frac{1+z}{1-z}))}+ix\frac{\pi}{2}+p\lambda x\log\left(i\frac{1+z}{1-z}\right)\right\}\right|\\
			&\le e^{px|\mathrm{Im}\lambda|\pi}\exp\left(p|\lambda|\left|\frac{1+z}{1-z}\right|^{x}\cos\left(\frac{\pi x}{2}+\arg(\lambda)+x\arg\left(\frac{1+z}{1-z}\right)\right)\right)\left|\frac{1+z}{1-z}\right|^{px\mathrm{Re}\lambda}.
		\end{align*}
		Since $\lambda\in i\overline{S(0,\pi-x\pi)}$, we have $\arg\lambda\in[\pi/2,(3/2-x)\pi]$. In this case, we understand that $\cos\left(\frac{\pi x}{2}+\arg(\lambda)+ x \arg\left(\frac{1+z}{1-z}\right)\right)\le0$, for all $z\in\mathbb{D}$, and as a result 
		\begin{equation}\label{eq:example2}
			\int_0^{2\pi} \left|e^{\lambda h(re^{i\theta})}\right|^p d\theta \le\int_0^{2\pi}\left|\frac{1+re^{i \theta}}{1-re^{i\theta}}\right|^{p x \mathrm{Re}\lambda} d\theta
		\end{equation}
		for all $r\in(0,1)$. But the preceding integral is bounded with respect to $r$, whenever $-\frac{1}{px}<\mathrm{Re}\lambda\le0$. Taking supremums with respect to $r\in(0,1)$, in \eqref{eq:example2} we get that $e^{\lambda h}\in H^p$, for all $p\ge1$. In this sense, we deduce that 
		$$\left\{\lambda\in\mathbb{C}: \arg\lambda\in\left[\frac{\pi}{2},\left(\frac{3}{2}-x\right)\pi\right] \, \textup{and} \,  \mathrm{Re}\lambda\in\left(-\frac{1}{px},0\right]\right\}\subseteq\sigma_p$$
		and hence equality prevails.
	\end{example}
	
	\begin{remark}
		We close the article with one final remark. There exist certain known inclusions between the Banach spaces we examined during the course of this work. Indeed, $\mathcal{D}\subset H^p\subset A^p_\alpha$, for all $p\ge1$ and all $\alpha>-1$. Since the point spectrum in each space can be written as the set of all $\lambda\in\C$ such that $e^{\lambda h}$ belongs to the space, where $h$ is the respective Koenigs function, we have that
		\begin{equation}\label{eq:space inclusions}
			\sigma_{\mathcal{D}}\subseteq \sigma_p \subseteq \sigma_{p,\alpha},\quad \textup{for all }p\ge1 \textup{ and all }\alpha>-1.
		\end{equation}
		Therefore, using all the lemmas and theorems described before, finding the point spectrum with respect to one of the spaces might aid in finding the point spectrum with respect to other spaces as well.
	\end{remark}

\section*{Acknowledgements}
	We thank Carlos G\'{o}mez-Cabello and F. Javier Gonz\'{a}lez-Do\~{n}a for sharing their manuscript \cite{carlos-javi} with us following the release of the IWOTA 2025 conference book of abstracts, where part of the results of this work were presented.

\end{document}